\RequirePackage{fix-cm}
\documentclass[smallextended]{svjour3}       
\smartqed  
\usepackage{epsfig,amsfonts,amssymb,latexsym,amsmath}
\usepackage{graphics}
\usepackage{subfigure}
\usepackage{graphicx}
\usepackage{epsfig}
\usepackage{algorithm}
\usepackage{algorithmic}
\begin{document}

\title{A Non-monotone Linear Search Method with Mixed Direction on Stiefel Manifold
}


\author{Harry Oviedo         \and
         Hugo Lara           \and
         Oscar Dalmau
}


\institute{Harry Fernando Oviedo Leon \at
              Mathematics Research Center, CIMAT A.C. Guanajuato, Mexico\\
              Tel.: +524737561760\\
              \email{harry.oviedo@cimat.mx}           
           \and
           Hugo Jos\'e Lara Urdaneta \at
           Universidade Federal de Santa Catarina. Campus Blumenau. Brasil\\
           Tel.: +5547932325145\\
           \email{hugo.lara.urdaneta@ufsc.br}
           \and
           Oscar Susano Dalmau Cede\~{n}o  \at
           Mathematics Research Center, CIMAT A.C. Guanajuato, Mexico \\
           Tel.: +524731201319\\
           \email{dalmau@cimat.mx}
}

\date{Received: date / Accepted: date}

\maketitle

\begin{abstract}
In this paper, we propose a non-monotone line search method for solving optimization problems on Stiefel manifold. Our method uses as a search direction a mixed gradient based on a descent direction, and a Barzilai-Borwein line search. Feasibility is guaranteed by projecting each iterate on the Stiefel manifold, through SVD factorizations. Some theoretical results for analyzing the algorithm are presented. Finally, we provide numerical experiments comparing our algorithm with other state-of-the-art procedures.
\end{abstract}

\section{Introduction}
\label{sec:1}
In this paper we consider optimization in matrices with orthogonality constraints,
\begin{equation}
  \min_{X\in \mathbb{R}^{n\times p}} \mathcal{F}(X) \quad s.t. \quad X^{\top}X = I  \label{problem}
\end{equation}
where $\mathcal{F}:\mathbb{R}^{n\times p}\rightarrow \mathbb{R}$ is a differentiable function and $I\in \mathbb{R}^{p\times p}$ is the identity matrix.
The feasible set  $ St(n,p) :=\{ X\in \mathbb{R}^{n\times p} / X^{\top}X=I \} $ is known in the literature as
the ``Stiefel Manifold'' which is reduced to the unit sphere when $p=1$ and in the case $p=n$ it is known as
``Orthogonal group'' (see \cite{Absil}). It is known that the dimension of the Stiefel manifold is $np- \frac{1}{2}p(p+1)$ \cite{Absil}, and it can be seen as an
embedded sub-manifold of $\mathbb{R}^{n\times p}$.  Problem (\ref{problem})  encompasses many applications such as nearest low-rank correlation matrix problem \cite{Grubi,Pietersz,Rebonato}, linear eigenvalue problem \cite{Golub,Saad}, Kohn-Sham total energy minimization \cite{YangMeza}, orthogonal procrustes problem \cite{Elden,Schonemann}, weighted orthogonal procrustes problem \cite{Francis}, sparse principal component analysis \cite{Ghaoui,Journee,Zou}, joint diagonalization~(blind source separation) \cite{Joho,Theis}, among others. In addition, many problems such as PCA, LDA, multidimensional scaling, orthogonal neighborhood preserving projection can be formulated as problem  (\ref{problem}) \cite{Kokiopoulou}.
On the other hand, it is known that  Stiefel manifold is a compact set which guarantees a global optimum, nevertheless it's not a convex set which  makes problem-solving (\ref{problem}) very hard. In particular, the \emph{quadratic assignment problem} (QAP), which can be formulated as minimization over a permutation matrix in $\mathbb{R}^{n\times n}$, that is, $X^{\top}X = I$, $X\geq 0$ and ``leakage interference minimization'' \cite{LiuDai} are NP-hard.\\

The problem (\ref{problem}) can be treated as a regular (vector) optimization problem, which conduces to several difficulties because orthogonal constraints can lead to many local minimizers. On the other hand, generating sequences of feasible points is not easy because preserving orthogonality constraints is numerically expensive. Most existing algorithms that generate feasible points either use routines to matrix reorthogonalization or generate points along geodesics on the manifold. The first one requires singular value decomposition or QR-decomposition, and the second one computes the matrix exponential or solves partial differential equations. We shall use the first approach. Some effort to avoid these approaches have been done,  calculating inverse matrices instead of SVD's at each step \cite{LaiOsher,WenYang,WenYin}.\\

By exploiting properties of the Stiefel Manifold, we propose a non-monotone line search constraint preserving algorithm to solve problem (\ref{problem}) with a mixed gradient based search direction. We analyze the resulting algorithm, and compare it to other state-of-the-art procedures, obtaining promising numerical results.

\subsection{Notation}
\label{subsec:1}
We say that a matrix $W\in \mathbb{R}^{n\times n}$ is skew-symmetric if $W = -W^{\top}$. The trace of $X$ is defined as the sum of the diagonal elements which we will denote by $Tr[X]$. The Euclidean inner product of two matrices $A,B\in \mathbb{R}^{m\times n}$ is  defined
as $\langle A, B\rangle : = \sum_{i,j}a_{i,j}b_{i,j} = Tr[A^{\top}B]$. The Frobenius norm is defined by the above inner product,
that is $||A||_F = \sqrt{\langle A,A\rangle}$. Let $\mathcal{F}:\mathbb{R}^{n\times p} \rightarrow \mathbb{R}$ be a differentiable function, we denote by $G := \mathcal{DF}(X):=( \frac{\partial \mathcal{F}(X)}{\partial X_{ij}} )$ the matrix of partial derivatives of $\mathcal{F}$. When evaluated at the current iterate $X_k$ of the algorithm, we denote the gradient by $G_k = \mathcal{DF}(X_k)$. Other objects depending on $X_k$ are also denoted with subindex $k$. Additionally, the directional derivative of $\mathcal{F}$ along a given matrix $Z$ at a given point $X$ is defined by:
\begin{eqnarray}
  \mathcal{DF}(X)[Z] := \lim_{t\rightarrow 0}\frac{\mathcal{F}(X+tZ) - \mathcal{F}(X)}{t} = \langle \mathcal{DF}(X),Z \rangle. \label{DerivDireccional}
\end{eqnarray}

\subsection{Organization}
\label{subsec:2}
The rest of the paper is organized as follows. In section 2 we propose and analyze our mixed constraint preserving updating scheme. Subsubsection 2.1 establishes the optimality conditions, while subsection 2.2 is devoted to the projection operator on the Stiefel Manifold. The proposed updating scheme is introduced and analyzed in subsection 2.3. Two algorithms are presented in section 3. The first one uses Armijo line search, and the second one Barzilai-Borwein stepsize scheme. Finally, numerical experiments are carried out for comparing our algorithm with other state-of-the-art procedures by solving instances of Weighted Orthogonal Procrustes Problem (WOPP), Total Energy minimization, and Linear Eigenvalue problems, are presented in section 4.

\section{A mixed constraint preserving updating scheme}
\label{sec:2}
\subsection{Optimality conditions}
\label{subsec:2}


The classical constrained optimization theory for continuous functions in $\mathbb{R}^n$ involve finding minimizers of the Lagrangean function applied to problem (\ref{problem}), given by:
\begin{eqnarray}
  \mathcal{L}(X,\Lambda)= \mathcal{F}(X) - \frac{1}{2} Tr[\,\Lambda^{\top}(X^{\top}X-I)\,], \label{funlagrange} \nonumber
\end{eqnarray}
where $\Lambda$ is the Lagrange multipliers matrix, which is a symmetric matrix because the matrix  $X^{\top}X$ is also symmetric.
The Lagrangean function leads to the first order optimality conditions for problem (\ref{problem}):

    \begin{equation}
    G - X\Lambda     =  0, \label{KKT1}
    \end{equation}
    \begin{equation}
    X^{\top}X - I   =  0. \label{KKT2}
    \end{equation}

By differentiating both sides of (\ref{KKT2}), we obtain the tangent space of Stiefel manifold at $X$:
$$
T_XSt(n,p) := \{ Z\in\mathbb{R}^{n\times p} : X^{\top}Z + Z^{\top}X = 0 \}.
$$

The following technical result, demonstrated by Wen-Yin in \cite{WenYin}, provides a tool to calculate roots for Eqs.~(\ref{KKT1})-(\ref{KKT2}).

\begin{lemma}\label{lema1}
   Suppose that  $X$ is a local minimizer of problem (\ref{problem}). Then $X$ satisfies the first order optimality
   conditions with the associated Lagrangian multiplier  $\Lambda = G^{\top}X$. Define
   \[
   \nabla \mathcal{F}_1(X) :=G-XG^{\top}X \quad and \quad A:=GX^{\top}-XG^{\top}
   \]

then $\nabla \mathcal{F}_1(X) = AX$. Moreover, $\nabla \mathcal{F}_1 = 0$ if and only if  $A=0$.
\end{lemma}

The lemma \ref{lema1} establishes an equivalence to the Eq.~(\ref{KKT1}) condition, i.e. $X\in\mathbb{R}^{n\times p}$ satisfies $\nabla \mathcal{F}_1(X)=0$ if and only if $X$ satisfies equation (\ref{KKT1}), so we can use this result as a stopping criterion for our algorithm.\\

\subsection{Tools on Stiefel Manifold}
\label{subsec:3}
The following lemma establishes an important property for matrices on Stiefel Manifold.
\begin{lemma}\label{lema_normfro}
  If $X\in St(n,p)$ then $||X^{\top}M||_F^{2}\leq||M||_F^{2}$ for all $M\in\mathbb{R}^{n\times p}$ matrix.
\end{lemma}
\begin{proof}
  In fact, let $X\in St(n,p)$ and $M\in\mathbb{R}^{n\times p}$, and let $X^{\top} = U_{X}\Sigma_{X}V_{X}^{\top}$, $M = U_{M}\Sigma_{M}V_{M}^{\top}$ be the singular value decomposition  of $X^{\top}$ and $M$ respectively, where $\Sigma_{X} = diag(\sigma_1,\sigma_2,\ldots,\sigma_p) \in \mathbb{R}^{p\times n}$ and $\Sigma_{M} = diag(\sigma_1^{M},\sigma_2^{M},\ldots,\sigma_p^{M}) \in \mathbb{R}^{n\times p}$ are diagonal matrices. We denote by
$Q = (q_{ij}) = V_{X}^{\top}U_{M}$. Since $Q\in\mathbb{R}^{n\times n}$ is a orthogonal matrix, i.e., $Q^{\top}Q = I = QQ^{\top}$ then,

\begin{equation}
  \sum_{i=1}^{n} q_{ij}^{2} = 1, \quad \forall j= 1,\ldots,n.   \label{prop2_eq3}
\end{equation}

We denote by $Q_p = Q(1:p,:)\in\mathbb{R}^{p\times n}$ the matrix obtained from $Q$ by extracting the last $n-p$ rows, consequently, $Q_pQ_p^{\top} = I$. Now, we have,

\begin{eqnarray}
  ||X^{\top}M||_F^{2} & = & ||U_{X}\Sigma_{X}V_{X}^{\top}U_{M}\Sigma_{M}V_{M}^{\top}||_F^{2} \label{prop2_eq4} \nonumber \\
                          & = & ||\Sigma_{X}Q\Sigma_{M}||_F^{2} \label{prop2_eq5} \\
                          & = & ||Q_p\Sigma_{M}||_F^{2} \label{prop2_eq6} \\
                          & = & \sum_{j=1}^{p}\Big[(\sigma_j^{M})^2(\sum_{i=1}^{p}q_{ij}^2)\Big] \label{prop2_eq7} \nonumber \\
                          & \leq & \sum_{j=1}^{p}(\sigma_j^{M})^2(\sum_{i=1}^{n}q_{ij}^2) \label{prop2_eq8} \\
                          & = & \sum_{j=1}^{p}(\sigma_j^{M})^2 \label{prop2_eq9} \\
                          & = & ||M||_F^{2}. \label{prop2_eq10} \nonumber
\end{eqnarray}
The second line of (\ref{prop2_eq5}) is obtained by using the fact that the Frobenius norm is invariant under orthogonal transformations, the third line (\ref{prop2_eq6}) is obtained by using that $\sigma_j = 1,\, \forall j\in\{1,\ldots,p\}$ because $X^{\top}X = I$  and the sixth line  (\ref{prop2_eq9}) is followed by using  (\ref{prop2_eq3}) in (\ref{prop2_eq8}). Thus, we conclude that

$$||X^{\top}M||_F^{2} \leq ||M||_F^{2}, \quad \forall M\in\mathbb{R}^{n\times p}.$$

\end{proof}

Another tool we shall employ is the projector operator on Stiefel manifold. First we define it and then provide a characterization, shown in \cite{Manton02}, in terms of the singular value decomposition of the underlined matrix.

\begin{definition}
Let $X\in \mathbb{R}^{n\times p}$ be a rank $p$ matrix. The projection operator $\pi:\mathbb{R}^{n\times p} \rightarrow St(n,p)$
is defined to be
\begin{equation}
  \pi(X) = \textrm{arg} \min_{Q\in St(n,p)} ||X - Q||_F^{2}. \label{projection}
\end{equation}
\end{definition}

The following proposition, demonstrated in \cite{Manton02}, provides us an explicit expression of the projection operator $\pi(\cdot)$.
\begin{proposition}
  Let $X\in \mathbb{R}^{n\times p}$ be a rank $p$ matrix. Then, $\pi(X)$
is well defined. Moreover, if the SVD of $X$ is $X = U\Sigma V^{\top}$, then $\pi(X) = UI_{n,p}V^{\top}$.
\end{proposition}

\subsection{Update scheme}
\label{sec:3}

In \cite{Manton02} it is presented an algorithm which adapts the well known steepest descent algorithm to solve problem (\ref{problem}). The ingredients of the updating formula are the derivative of the Lagrangean function, the projection operator (\ref{projection}) and the step size choice by means of Armijo. The direction in (\ref{direction}) is in the tangent space of the manifold at the current point.
With the intuition of perturbing the steepest descent direction, we propose a modification of Manton's procedure by a mixture of tangent directions which incorporates $\nabla \mathcal{F}_2=(I-XX^{\top})G$. Note that  if $X^*$ is a local minimizer of the problem (\ref{problem}) then $\nabla \mathcal{F}_1(X^*)=\nabla \mathcal{F}_2(X^*)$. Since both of the directions belong to $T_XSt(n,p)$, then the obtained mixture is also in this space. This mean that the obtained algorithm preserve the matrix structure of the problem, instead of using optimization in $\mathbb{R}^n$.\\

In this paper we focus on a modification of the projected gradient-type methods: given a feasible point $X$, we compute the new iteration $Z(\bar\tau)$ as a point on the curve:
\begin{eqnarray}
 Z(\tau)   & = & \pi(X - \tau H), \label{update_formulae}
\end{eqnarray}
where the term $\tau>0$ represents the step size. The direction $H$ is defined by:
\begin{equation}
H := \alpha\nabla \mathcal{F}_1(X) + \beta \nabla \mathcal{F}_2(X),  \label{direction}
\end{equation}
where $\alpha>0$, $\beta\geq0$, $\nabla \mathcal{F}_1(X)$ is defined as in lemma \ref{lema1} and $\nabla \mathcal{F}_2(X)$ is given by:
\begin{eqnarray}
\nabla \mathcal{F}_2(X) &:=& (I - XX^{\top})G.  \nonumber
\end{eqnarray}

Note that $Z(0)=X$, and the new trial point $Z(\bar\tau)$ satisfies the orthogonality constraints of the problem (\ref{problem}) because  $Z(\tau)$ is a curve on the manifold. Proposition \ref{prop1} below shows that direction $H$ belongs to the tangent space of the manifold at the point $X$.

\begin{proposition}\label{prop1}
The direction matrix defined in (\ref{direction}) belongs to the tangent space of $St(n,p)$ at $X$.
\end{proposition}

\begin{proof}
We must prove that: $H\in T_{X}St(n,p)$.
For this, we prove that: $W_1 := \nabla \mathcal{F}_1(X) \in T_{X}St(n,p)$ and $W_2 := \nabla \mathcal{F}_2(X)\in T_{X}St(n,p)$. In fact, by using $A^{\top}=-A$, we have,
\begin{eqnarray}
W_1^{\top}X + X^{\top}W_1 & = & X^{\top}A^{\top}X + X^{\top}AX  \nonumber \\
                              & = & 0, \nonumber
\end{eqnarray}
and due to the feasibility of $X$ ($X^{\top}X=I$) we obtain
\begin{eqnarray}
W_2^{\top}X + X^{\top}W_2 & = & (\, G^{\top}X - G^{\top}X X^{\top}X \,) + (\, X^{\top}G - X^{\top}X X^{\top}G \,) \nonumber \\
                              & = & 0. \nonumber
\end{eqnarray}
Consequently $W_1$ and $W_2$ belongs to $T_XSt(n,p)$. Since $T_{X}St(n,p)$ is a vector space, we have the linear combination of $W_1$, and $W_2$ also belongs to $T_{X}St(n,p)$, concluding that $H\in T_{X}St(n,p)$.
\end{proof}

The following proposition shows that the curve $Z(\tau)$ defines a descent direction for $\alpha>0$. $\alpha$ values next to zero provide bad mixed directions.

\begin{proposition}\label{prop2}
If $\alpha>0$ and $\beta \ge 0$ then $Z(\tau)$ is a descent curve at $\tau=0$, that is
\begin{equation}
\mathcal{DF}(X)[\dot{Z}(0)] = \frac{\partial \mathcal{F}(X+\tau \dot{Z}(0))}{\partial \tau} \Big|_{\tau = 0} \leq  - \frac{\alpha}{2}||A||_F^{2}  < 0. \nonumber
\end{equation}
\end{proposition}

\begin{proof}
We begin by calculating the derivative of the curve $Z(\tau)$ at $\tau = 0$. From Taylor's second order approximation in Stiefel manifold (prop. 12 in \cite{Manton02}), if $X\in St(n,p)$ and $U\in T_{X}St(n,p)$ then:
\begin{equation}
\pi(X+\tau U) = X+\tau U - \frac{\tau^{2}}{2}XU^{\top}U + O(\tau^{3}), \label{prop12}
\end{equation}
so, deriving (\ref{prop12}) with respect to $\tau$, and evaluating at $\tau=0$,
\[ \frac{\partial \pi(X+\tau U)}{\partial \tau}(0)= U.\]
It follows from this fact, and our update formula $Z(\tau)$ that,
\[ \dot{Z}(0) = -H. \]
Now, from the definition (\ref{DerivDireccional}) and using trace properties we have,
\begin{eqnarray}
\mathcal{DF}(X)[\dot{Z}(0)] &=& -Tr[G^{\top}H] \nonumber \\
&=& - \alpha Tr[G^{\top}A X] - \beta Tr[G^{\top}(G - XX^{\top}G)] \nonumber \\
& = &  - \frac{\alpha}{2}||A||_F^{2} - \beta||G||_F^{2} + \beta Tr([(X^{\top}G)^{\top}X^{\top}G] \nonumber \\
& = & - \frac{\alpha}{2}||A||_F^{2} - \beta||G||_F^{2} + \beta ||X^{\top}G||_F^{2}, \label{prop2_eq2} \nonumber
\end{eqnarray}

since $\alpha>0$ and using lemma \ref{lema_normfro} we arrive at:
\begin{eqnarray}
\mathcal{DF}(X)[\dot{Z}(0)] & \leq & - \frac{\alpha}{2}||A||_F^{2}  \nonumber \\
                              & < & 0, \nonumber
\end{eqnarray}
which completes the proof.
\end{proof}

From proposition \ref{prop1}, we obtain that the mapping $Z(\tau)$ defined in (\ref{update_formulae}) is a retraction on the Stiefel manifold, see \cite{Absil,Absil2}. Hence, the convergence results in \cite{Absil} regarding retractions apply directly to Algorithm \ref{Alg1}.

\section{Strategies to select the step size}
\label{sec:4}
\subsection{Armijo condition}
\label{subsec:4}
It is well known that the steepest descent method with a fixed step size may not converge. However,
by choosing the step size wisely, convergence can be guaranteed and its speed can be accelerated without
significantly increasing of the cost at each iteration. At iteration $k$, we can choose a step size by minimizing
$\mathcal{F}(Z_k(\tau))$ along the curve $Z_k(\tau)$ with respect to $\tau$. Since finding its global minimizer is computationally expensive,
it is usually sufficient to obtain an approximated minimum, in order to deal with this, we use the Armijo condition \cite{Nocedal}
to select a step size:
\begin{equation}
  \mathcal{F}(Z_k(\tau)) \leq \mathcal{F}(X_k) + \rho_1\tau \mathcal{DF}(X_k)[\dot{Z}_k(0)] \label{Armijo1} \nonumber
\end{equation}
with $0<\rho_1<1.$\\

Our approximated monotone procedure is resumed in algorithm \ref{Alg1}.

\begin{algorithm}[H]
\begin{algorithmic}[1]
\REQUIRE $X_{0}\in St(n,p)$, $\tau_0>0$, $\sigma,\epsilon,\delta\in(0,1)$, $X_{-1} = X_{0}$, $k = 0$.\label{lin:lineaRara}
\ENSURE $X^{*}$ a local minimizer.
\WHILE{ $|| \nabla \mathcal{F}(X_k) ||_F >\epsilon$ }
\STATE $\tau = \tau_0$,
\WHILE { $\mathcal{F}(Z_k(\tau)) \geq \mathcal{F}(X_k) + \rho_1\tau \mathcal{DF}(X_k)[\dot{Z}_k(0)]$}
\STATE $\tau = \delta\tau$,
\ENDWHILE
\STATE $X_{k+1} = Z_k(\tau)$,
\STATE $k = k+1$,
\ENDWHILE
\STATE $X^{*} = X_{k}.$
\end{algorithmic}
\caption{Monotone Algorithm}\label{Alg1}
\end{algorithm}

\subsection{Nonmonotone search with Barzilai Borwein step size}
\label{subsec:5}
We propose a variant of algorithm \ref{Alg1}, instead of using Armijo condition, we acoplate Barzilai Borwein (BB-step) step size, see \cite{BBstep1}, which sometimes  improve the performance of linear search algorithms such as the steepest descent method without adding a lot of extra computational cost. This technique considers the classic updating of the line search methods:
\[
X_{k+1} = X_k - \alpha \nabla F(X_k),
\]
where $\nabla F$ is the gradient of the objective function, and $\alpha$ is the step size. This approach (Barzilai Borwein step size) proposes as the step size, the value $\alpha$ that satisfies the secant equation:
 \begin{eqnarray}
\alpha_k^{BB1} &:=& \textrm{argmin}_{\alpha > 0 }||B(\alpha)S_k - R_k ||, \label{bb1}
\end{eqnarray}
or well,
\begin{eqnarray}
\alpha_k^{BB2}  &:= & \textrm{argmin}_{\alpha> 0}||S_k - B(\alpha)^{-1}R_k ||, \label{bb2}
\end{eqnarray}
where $S_k = X_{k+1}-X_k$,  $R_k = \nabla F(X_{k+1}) - \nabla F(X_k)$ and the matrix $B(\alpha) = (\alpha I)^{-1}$, is considered an approximation of the Hessian of the objective function, so the step size  $\alpha$ is obtained by forcing a Quasi-Newton property. It follows from Eqs.~(\ref{bb1})-(\ref{bb2}) that,
\begin{equation}
\alpha_k^{BB1} = \frac{||S_k||_F^2}{Tr[S_k^{\top}R_k]}\quad  \textrm{and} \quad   \alpha_k^{BB2} = \frac{Tr[S_k^{\top}R_k]}{||R_k||_F^2}. \label{BB-steps}
\end{equation}
Since the quantities $\alpha_k^{BB1}$ and $\alpha_k^{BB2}$ can be negatives, it is usually taken the absolute value of any of these step sizes. On the other hand, the BB steps do not necessarily guarantee the descent of the objective function at each iteration, this may imply that the method does not converge. In order to solve this problem we will use a globalization technique  which guarantees global convergence on certain conditions \cite{DaiF,Raydan2}, that is, we use a  non-monotone linear search described as in \cite{ZhangHager}. From the above considerations we arrive at the following algorithm:

\begin{algorithm}[H]
\begin{algorithmic}[1]
\REQUIRE $X_{0}\in St(n,p)$, $\tau>0$, $0<\tau_m \leq \tau_M$, $\rho_1,\epsilon,\eta,\delta \in (0,1)$, $Q_0 = 1$, k=0. \label{lin:lineaRara}
\ENSURE $X^{*}$ a local minimizer.
\WHILE{ $|| \nabla \mathcal{F}(X_k) ||_F >\epsilon$ }
\WHILE { $\mathcal{F}(Z_k(\tau)) \geq C_k + \rho_1\tau \mathcal{DF}(X_k)[\dot{Z}_k(0)]$}
\STATE $\tau = \delta\tau$,
\ENDWHILE
\STATE $X_{k+1} = Z_k(\tau)$, $Q_{k+1} = \eta Q_k + 1$\,and\, $ C_{k+1} = (\eta Q_kC_k + \mathcal{F}(X_{k+1}) )/Q_{k+1}$,
\STATE choose $\tau = \alpha_k^{BB1}$ or $\tau = \alpha_k^{BB2}$ with $\alpha_k^{BB1},\alpha_k^{BB2}$ defined as in (\ref{BB-steps})
\STATE $ \tau = \max(\min( \tau, \tau_M ),\tau_m )$,
\STATE $k = k+1$,
\ENDWHILE
\STATE $X^{*} = X_{k}.$
\end{algorithmic}
\caption{Non-monotone Algorithm}\label{Alg2}
\end{algorithm}

Note that when $\eta=0$ the algorithm \ref{Alg2} becomes algorithm \ref{Alg1} with BB-step. Moreover, when we select the parameters $\alpha = 1$ and $\beta = 0$, we obtain a procedure very similar to the ``Modified Steepest Descent Method'' (\emph{MSDStifel}) proposed by Manton in \cite{Manton02}, however, in this case, our algorithm \ref{Alg2} is an accelerated version of \emph{MSDStifel}, since it incorporates a non-monotone search combined with the BB-step, which usually accelerates the gradient-type methods, whereas the algorithm presented in \cite{Manton02}, uses a double backtracking strategy to estimate the step size, that in practice is very slow since requires more computing.\\

In addition, in our implementation of the algorithm \ref{Alg2}, we update $X_{k+1}$ by the approximation (\ref{prop12}), specifically, we calculate
$\hat{X}_{k+1} = X_k-\tau_k H_k - \frac{\tau_k^{2}}{2}X_kH_k^{\top}H_k$, with $H_k$ as in (\ref{direction}), and in case the feasibility error is sufficiently small, that is, $||\hat{X}_{k+1}^{\top}\hat{X}_{k+1} - I||_F <$ 1e-13 this point is accepted and we update the new point by $X_{k+1} = \hat{X}_{k+1}$; otherwise, we update the new trial point by $X_{k+1} = U(:,1:p)V^{\top}$ where $[U,\Sigma,V] = \verb"svd"(X_k - \tau_k H_k, 0)$ using Matlab notation. Note that if the step size $\tau_k$ is small, or if the sequence $\{X_{k}\}$ approaches a stationary point of the Lagrangian function
(\ref{funlagrange}) then (\ref{prop12}) closely approximates the projection operator, in view of this, in several iterations our algorithm may saves the SVD's computation.\\

All these details make our algorithm \ref{Alg2} into a quicker and improved version of the \emph{MSDStifel} algorithm, for the case when the parameters
$\alpha = 1$ and $\beta = 0$; and in the case when we take another different selection of these parameters our method incorporates a mixture of descent directions that has shown to be a better direction that the one used by Manton, in some cases. In the section of experiments (see table \ref{tab:2}), we compare our algorithm \ref{Alg2} (with $\alpha = 1 $ and $\beta = 0$) against \emph{MSDStifel}, in this experiment it is clearly shown that our algorithm is much faster and efficient than the one proposed by Manton.

\section{Numerical experiments}
\label{sec:5}
In this section, we will study the performance of our approaches to different optimization problems with orthogonality constraints, and we show the efficiency and effectiveness of our proposals on these problems. We implemented both Algorithms 1 and 2 in MATLAB. Since
Algorithm 2 appears to be more efficient in most test sets, we compare both algorithms on the first test set in subsection \ref{subsec:7} and compare only Algorithm 2 with other two state-of-the-art algorithms on problems in the remaining test. In all experiments presented in upcoming subsections, we used the default parameters given by the author's implementations solvers for the abovementioned algorithms. We specify in each case the values of $\alpha$ and $\beta$ for the mixed direction. When $\alpha=1$ and $\beta=0$, our direction coincide with Manton's direction, and for this case, the difference in the procedures is established by the line search. The experiments were ran using Matlab 7.10 on a intel(R) CORE(TM) i3, CPU 2.53 GHz with 500 Gb of HD and 4 Gb of Ram.

\subsection{Stopping rules}
\label{subsec:6}
In our implementation, we are checking the norm of the gradient, and monitoring the relative changes of two consecutive iterates and their corresponding objective function values:
\begin{equation}
  \textrm{rel}_k^{x} := \frac{||X_{k+1} - X_{k}||_F}{\sqrt{n}}, \quad \textrm{rel}_{k}^{f} := \quad \frac{|\mathcal{F}(X_{k}) - \mathcal{F}(X_{k+1})|}{|F(X_k)| + 1}. \nonumber
\end{equation}

Then, given a maximum number K of iterations, our algorithm will stop at iteration $k$ if $k<K$ or $||\nabla \mathcal{F}(X_k)||_F\leq \epsilon$, or $\textrm{rel}_k^{x}< tolx$ and $\textrm{rel}_k^{f}< tolf$, or
\begin{equation}
  \textrm{mean}( [\textrm{rel}_{1+k-min(k,T)}^{x},\ldots,\textrm{rel}_{k}^{x}] )\leq 10tolx \quad and \quad \textrm{mean}( [\textrm{rel}_{1+k-min(k,T)}^{f},\ldots,\textrm{rel}_{k}^{f}] )\leq 10tolf \nonumber
\end{equation}
The default values for the parameters in our algorithms are: $\epsilon =$ 1e-4, $tolx =$ 1e-6, $tolf =$ 1e-12, $T = 5$, $K = 1000$, $\delta = 0.3$, $\rho_1 =$ 1e-4, $\tau_m =$ 1e-20, $\tau_M =$ 1e+20 and $\eta = 0.85$.\\

In the rest of this section we will denote by: ``Nitr'' to the number of iterations, ``Nfe'' to the number of evaluations of the objective function, ``Time'' to CPU time in seconds, ``Fval'' to the value of the evaluated objective function in the optimum estimated, ``NrmG'' to the norm of Gradient of the lagrangean function evaluated in the optimum estimate ($||\nabla F_1(\hat{X})||_F$) and ``Feasi'' to the feasibility error ($||\hat{X}^{\top}\hat{X} - I||_F$), obtained by the algorithms.

\subsection{Weighted orthogonal procrustes problem (WOPP)}
\label{subsec:7}
In this subsection, we consider the Weighted Orthogonal Procrustes Problem (WOPP) \cite{Francis}, which is formulated as follows:
\begin{equation}
  \min_{X\in \mathbb{R}^{m\times n}} \frac{1}{2}||AXC - B||_F^{2} \quad s.t. \quad X^{\top}X = I \nonumber
\end{equation}
where $A\in \mathbb{R}^{p\times m}$, $B\in\mathbb{R}^{p\times q}$ and $C\in \mathbb{R}^{n\times q}$ are given matrices.\\

For the numerical experiments we consider $n=q$, $p=m$, $A = PSR^{\top}$ and $C = Q\Lambda Q^{\top}$, where $P$ and $R$ are random orthogonal matrices, $Q\in \mathbb{R}^{n\times n}$ is a Householder matrix, $\Lambda\in \mathbb{R}^{n\times n}$ is a diagonal matrix with elements uniformly distributed in the interval $[\frac{1}{2},2]$ and $S$ is a diagonal matrix defined for each type of problem. As a starting point $X_{0}$ we randomly generate the starting points by the built-in functions $\verb"randn"$ and $\verb"svd"$:
\begin{equation}
\overline{X} = \verb"randn" (m, n), \quad [U,\Sigma,V] = \verb"svd"(\overline{X},0), \quad X_{0} = UV^{\top}. \nonumber
\end{equation}
When it is not specified how the data were generated, it was understood that they were generated following a normal standard distribution.\\


Against the exact solution of the problem, we create a known solution $Q_{*}\in \mathbb{R}^{m\times n}$ by taking $B = AQ_{*}C$. The tested problems were taken from \cite{Francis} and are described below.\\

\textbf{Problem 1}: The diagonal elements of $S$ are generated by a normal truncated distribution in the interval [10,12].\\

\textbf{Problem 2}: The diagonal of $S$ is given by $S_{ii} = i + 2r_{i}$, where $r_i$ are random numbers uniformly distributed in the interval [0,1].\\

\textbf{Problem 3}: Each diagonal element of $S$ is generated as follows: $S_{ii} = 1 + \frac{99(i-1)}{m+1} + 2r_{i}$, with $r_i$ uniformly distributed in the interval [0,1].\\

\begin{figure}
  \centering
  \begin{center}
  \subfigure[Iteration number vs $\alpha$, Problem 1]{\includegraphics[width=5cm]{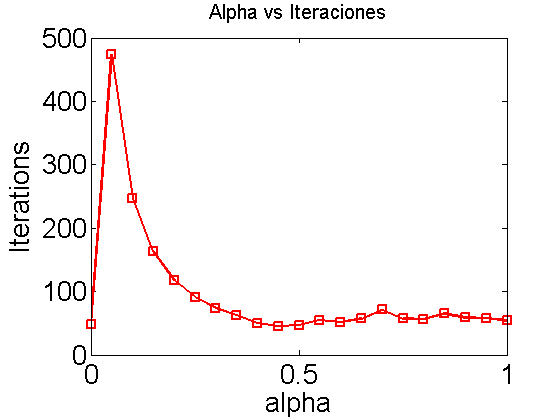}}  \hspace{1.5cm}
  \subfigure[Iteration number vs $\alpha$, Problem 2]{\includegraphics[width=5cm]{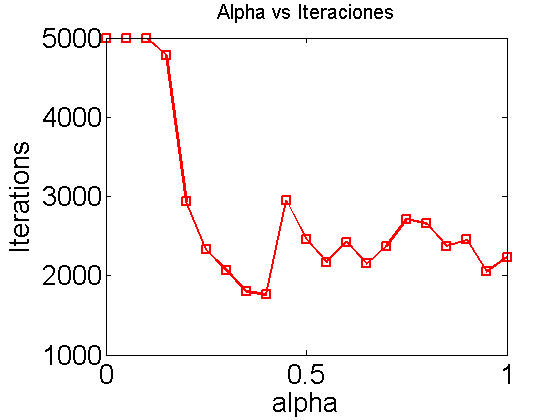}}  \hspace{1.5cm}
  \subfigure[Iteration number vs $\alpha$, Problem 3]{\includegraphics[width=5cm]{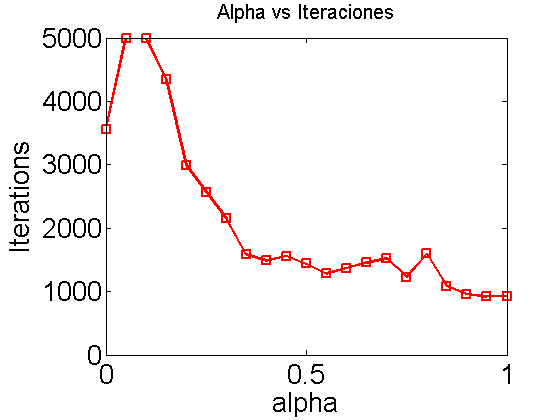}}  \hspace{1.5cm}
  \end{center}
  \caption{Iteration number vs $\alpha$ with $m = 100$ and $n = 65$}
  \label{fig:0}
\end{figure}

Note that when the matrix $S$ is generated following problem 1 then the matrix $A$ has a small condition number, and when $S$ is generated following problems 2 and 3, the matrix $A$ has a big condition number, which becomes bigger as $m$ grows. In view of this, in the remainder of this subsection, we will refer to \textbf{Problem 2} and \textbf{Problem 3} as ill conditioned WOPP's problems and to \textbf{Problem 1} as well conditioned problems.\\

First of all, we perform an experiment in order to calibrate the parameters ($\alpha$,$\beta$). To make the study more tractable, we consider the direction $H_k$ as a convex combination of $\nabla\mathcal{F}_1(X_k)$ and 
$\nabla\mathcal{F}_2(X_k)$, that is, we put $\beta= 1-\alpha $ with $\alpha \in [0,1]$. More specifically, we randomly generated one problem of each type, as explained above (\textbf{Problem 1}, \textbf{Problem 2} and \textbf{Problem 3} are generated) selecting $m = 100$ and $n=65$, and each one is solved for the following alpha values: $\alpha = 0: 0$.$05 :1$ (using Matlab notation). Figure \ref{fig:0} shows the curves of the iterations versus each alpha value, and for each type of problem. In this figure we see that for well-conditioned problems (\textbf{Problem 1}) our algorithm obtains the best result when it uses $\alpha = 0.45$, converging in a time of 0.1388 seconds and with gradient norm of $NrmG = 9.5410e-05 $. In this plot it seems that $ \alpha = 0$ also produces good results because it performs very few iterations, however the algorithm gets a bad result in terms of NrmG. On the other hand, we see that for ill-conditioned problems (\textbf{Problem 2,3}) the algorithm \ref{Alg2} obtains better results when $\alpha$ approaches 1, in particular in the plot we notice that with $\alpha = 1$ less iterations are done.\\

From this experiment and our experience testifying our algorithm for different values of $\alpha $, we note that for well-conditioned WOPP problems, our algorithm tends to perform better with values of $\alpha $ close to 0.5; whereas for WOPP ill-conditioned problems our procedure shows best results when $\alpha $ is close to 1. However, it will remain as future work to study the performance of our method for the case when the direction $H_k$ is taken as a linear combination instead of a convex combination of $\nabla\mathcal{F}_1(X_k) $ and $\nabla\mathcal{F}_2(X_k) $.\\

It is worth mentioning that deciding which set of parameters to use to obtain a good performance of our algorithm is not an easy task to predict, since in general these parameters will depend to a great extent on the objective function and the good or ill conditioning of the problem. A strategy to select these parameters could be to use some metaheuristic that optimize these parameters for a specific problem or application. However, based on these experiments and in our numerical experience running our algorithm, we will use for the following subsections $\beta = 1- \alpha $ and we will take alpha in the set $\{0$.$5,1\} $ or some number close to 1, because this choice usually reach good results.\\

In the second experiment, we will test the efficiency of the non-monotone line search, by making a comparison between the \textbf{Algorithm 1} and the \textbf{Algorithm 2}. In the Table~\ref{tab:1} we present the results of this comparison, in this experiment we choose $(\alpha,\beta) = (1,0)$ in both algorithms. We show the minimum (min), maximum (max) and the average (mean) of the results achieved from $K = 30$ simulations. According to the number of iterations a small difference in favor to algorithm 1 can be observed. However, in terms of CPU time, the dominance of algorithm 2 is overwhelming. This conclusion is also appreciated in figure \ref{fig:1}. In view of this, for the remaining experiments we only compare our algorithm 2, which we hereafter refer as \textbf{Grad-Retract}, to other state of the art procedures, namely Edelman, Manton, Wen-Yin \cite{Edelman,Manton02,WenYin}.\\


\begin{table}
\caption{\small{A comparative between the monotone and non-monotone case of Grad-retrac method}}
\label{tab:1}
\resizebox{12cm}{!}{\begin{tabular}{|c|l|c|c|c|c|c|c|c|c|c|c|c|c|}
  \hline
  & &  \multicolumn{4}{|c|}{ Problem 1, $m = 500$}&  \multicolumn{4}{|c|}{ Problem 2, $m = 100$}&  \multicolumn{4}{|c|}{ Problem 3, $m = 100$ } \\
  \hline
  & & \multicolumn{2}{|c|}{$n=20$} & \multicolumn{2}{|c|}{$n=70$} & \multicolumn{2}{|c|}{$n=10$} & \multicolumn{2}{|c|}{$n=50$} & \multicolumn{2}{|c|}{$n=10$} & \multicolumn{2}{|c|}{$n=50$} \\
  \hline
                 &         & \textbf{Nitr} & \textbf{Time} & \textbf{Nitr} & \textbf{Time} & \textbf{Nitr}  & \textbf{Time}    & \textbf{Nitr}   & \textbf{Time}     & \textbf{Nitr}  & \textbf{Time} & \textbf{Nitr} & \textbf{Time} \\
  \hline
	                &min	&24	    &0.1876	&33	    &0.6347	&716	&0.3561	&612	  &1.6184	&579	&0.4152	&578	    &1.6255\\
\textbf{Algorithm 1}&mean	&35.23	&0.3056	&42.30	&0.8156	&1182.1	&0.8397	&1447.60  &3.5599	&1074.9	&0.6466	&1318    	&3.4910\\
	                &max	&44	    &0.5432	&49	    &0.9340	&2920	&2.0068	&6902	  &10.802	&3036	&1.7612	&5628	    &12.736\\
\hline
	                &min	&22	    &0.0899	&34	    &0.4564	&700	&0.3066	&594	  &1.5256	&630	&0.3467	&639	    &1.5509\\
\textbf{Algorithm 2}&mean	&35.23	&0.1574	&42.97	&0.5480	&1193.3	&0.6762	&1456.10  &3.0685	&1060.8	&0.5769	&1355.90	&3.1333\\
	                &max	&45	    &0.2446	&50	    &0.6152	&2862	&1.6512	&7702	  &11.130	&2322	&1.3188	&5084	    &10.349\\
  \hline
\end{tabular}}
\end{table}

\begin{figure}
  \centering
  \begin{center}
  \subfigure[Iteration number]{\includegraphics[width=5cm]{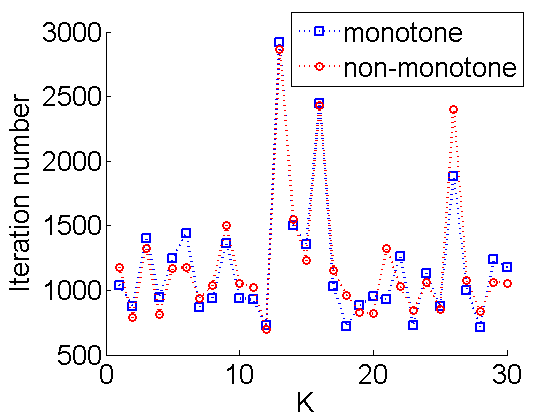}}  \hspace{1.5cm}
  \subfigure[CPU time in seconds]{\includegraphics[width=5cm]{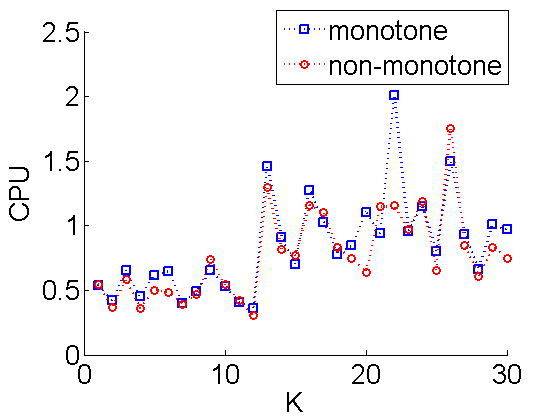}}
  \end{center}
  \caption{Algorithm 1 vs Algorithm 2, from Table \ref{tab:1}, solving Problem 2 with $m = 100$ and $n = 10$}
  \label{fig:1}
\end{figure}

\begin{table}
\centering
\caption{\small{A comparative between the Algorithm 2 and the Modified Steepest Descent Method proposed in \cite{Manton02}}}
\label{tab:2}
\resizebox{12cm}{!}{\begin{tabular}{|c|l|c|c|c|c|c|c|c|c|c|c|c|c|}
  \hline
  & &  \multicolumn{4}{|c|}{ Problem 1, $m = 50$}&  \multicolumn{4}{|c|}{ Problem 2, $m = 50$}&  \multicolumn{4}{|c|}{ Problem 3, $m = 50$ } \\
  \hline
  & & \multicolumn{2}{|c|}{$n=20$} & \multicolumn{2}{|c|}{$n=50$} & \multicolumn{2}{|c|}{$n=20$} & \multicolumn{2}{|c|}{$n=50$} & \multicolumn{2}{|c|}{$n=20$} & \multicolumn{2}{|c|}{$n=50$} \\
  \hline
                 &         & \textbf{Nitr} & \textbf{Time} & \textbf{Nitr} & \textbf{Time} & \textbf{Nitr}  & \textbf{Time}    & \textbf{Nitr}   & \textbf{Time}     & \textbf{Nitr}  & \textbf{Time} & \textbf{Nitr} & \textbf{Time} \\
  \hline
	                &min	&39	   &0.0969	&179	&0.9531	&1851	&19.811	&895	  &9.8746	&3959	&69.25	&2779	 &45.387\\
\textbf{MSDStiefel}	&mean	&54.70 &0.1317	&541.80	&4.0235	&4471.3	&111.30	&2836.9   &67.527	&6479.0	&186.0	&5611.70 &156.61\\
	                &max	&77	   &0.2208	&1494	&13.457	&7000	&239.69	&7000	  &270.02	&7000	&243.05	&7000	 &240.51\\
\hline
	                &min	&22	   &0.0181	&118	&0.1676	&369	&0.2808	&296	  &0.4930	&653	&0.4844	&417	&0.6182\\
\textbf{Grad-retrac}&mean	&25.73 &0.0221	&220.40	&0.3208	&635.4	&0.481	&548.47	  &0.8285	&1153.3	&0.8561	&748.70	&1.0956\\
	                &max	&30	   &0.0627	&444	&0.6487	&1461	&1.0983	&1319	  &1.6947	&2868	&2.2797	&1274	&1.8736\\
  \hline
\end{tabular}}
\end{table}

In the Table~\ref{tab:2}, we compare our \textbf{Grad-Retrac} versus the Modified Steepest Descent Method (MSDStiefel) proposed in \cite{Manton02}. More specifically, we compare the performance of these methods solving WOPP's. We create the matrices $A$, $C$ and $X_0$ as explained at the beginning of this subsection and generate the matrix $B$ randomly with their entries uniformly distributed in the range [0,1]. A total of 30 simulations were run. We see in Table \ref{tab:2} that the performance of our \textbf{Grad-Retrac} is consistently better than the performance of the algorithm proposed by Manton \cite{Manton02}, in terms of number of iterations, and CPU time. Both of the procedures solve all proposed test problems.\\

For all experiments presented in Tables \ref{tab:3}-\ref{tab:5}, we compare our algorithm \textbf{Grad-retrac} with the methods \textbf{PGST}, \textbf{OptStiefel} proposed in \cite{Francis}, \cite{WenYin}\footnote{The OptStiefel solver is available in http://www.math.ucla.edu/$\sim$wotaoyin/papers/feasible\_method\_matrix\_manifold.html} respectively, and we use as tolerance $\epsilon =$ 1e-5 and as maximum number of iterations $K = 8000$ for all methods. In addition, in the experiments shown in Table \ref{tab:3} we choose $\alpha = \beta = 0.5$ (see Eq.~\ref{direction}), while for the experiments shown in Tables \ref{tab:4}-\ref{tab:5}, we select: $\alpha = 1$, and $\beta = 0$ in the descent direction of our method. For each value to compare in these tables, we show the minimal (min), maximum (max) end the average (mean) of the results achieved from 30 simulations.\\

Table \ref{tab:3} include numerical results for well conditioned WOPP's of different sizes. We observe from this table that all methods converge to good solutions, while our method (\textbf{Grad-retrac}) always performs much better than \textbf{PGST} and \textbf{OptStiefel} in terms of the CPU time, except on orthogonal group problems (see table Ex.3 and Ex.4 in table \ref{tab:3}), where \textbf{OptStiefel} method shows better performance. Tables \ref{tab:4}-\ref{tab:5} presents the results achieved for all methods solving WOPP's in presence of ill conditioning. In these tables, we see that \textbf{Grad-retrac} and \textbf{OptStiefel} algorithms show similar performance in most experiments, while the PGST method exhibit the worst
performance in all cases examined here, however all methods achieve good solutions.\\

\begin{table}
\centering
\caption{\small{Performance of the methods for well conditioned WOPP problems}}
\label{tab:3}
\resizebox{12cm}{!}{\begin{tabular}{|l|c|c|c|c|c|c||c|c|c|c|c|}
  \hline
  \textbf{Methods} &	    &\textbf{Nitr}    &   \textbf{Time}	&\textbf{NrmG}	&\textbf{Fval}	&\textbf{Feasi}	&\textbf{Nitr}	 &\textbf{Time}	&\textbf{NrmG}	&\textbf{Fval}	&\textbf{Feasi} \\
  \hline
  \hline
  & & \multicolumn{5}{|c||}{ Ex.1: $n = 500$, $p = 70$, $problem = 1$  }& \multicolumn{5}{|c|}{ Ex.2: $n = 1000$, $p = 100$, $problem = 1$   }\\
  \hline			 																
		            &min	&51	   	&0.7151   &2.79e-05	&3.64e-12	&3.40e-14		&52	   	&3.4933  &9.92e-06    &1.98e-13  &4.21e-14\\
\textbf{OptStiefel} &mean	&56.70  &0.8595   &7.01e-05	&4.62e-11	&4.14e-14		&57.57  &4.7568  &9.15e-05    &5.37e-11  &5.15e-14\\	
		            &max	&64	   	&1.1443   &9.99e-05	&1.29e-10	&4.86e-14		&67	   	&7.3991  &8.71e-04    &2.05e-10  &5.97e-14\\	
\hline
		            &min	&35	   	&0.6698   &1.54e-05	&2.07e-12	&4.71e-15		&34	   	&3.2884  &2.84e-05	&2.66e-12  &5.83e-15\\
\textbf{PGST}       &mean	&37.73  &0.8310   &6.80e-05	&4.52e-11	&7.01e-15		&39.27  &4.8581  &6.65e-05	&3.44e-11  &9.42e-15\\
		            &max	&43	   	&1.1251   &9.82e-05	&1.34e-10	&9.65e-15		&44	   	&6.9226  &9.56e-05	&1.06e-10  &1.41e-14\\
\hline
		            &min	&37	   	&0.5178   &2.19e-05	&1.85e-12	&1.60e-14		&38	   	&2.4845  &1.35e-05	&6.34e-13  &2.05e-14\\
\textbf{Grad-retrac}&mean	&42.37  &0.7121   &6.86e-05	&4.11e-11	&3.24e-14		&43.03  &3.6456  &6.13e-05	&4.70e-11  &4.08e-14\\
		            &max	&48	   	&0.9896   &9.95e-05	&1.38e-10	&9.18e-14		&47	   	&5.5771  &9.97e-05	&1.51e-10  &9.88e-14\\
\hline
\hline
  & & \multicolumn{5}{|c||}{ Ex.3: $n = 200$, $p = 200$, $problem = 1$  }& \multicolumn{5}{|c|}{ Ex.4: $n = 300$, $p = 300$, $problem = 1$  }\\
\hline
		            &min	&46	   	&0.7488   &1.40e-05	&2.13e-13	&6.52e-14		&49	   	&2.3306  &7.14e-06    &8.04e-14  &1.41e-14\\
\textbf{OptStiefel}	&mean	&51.27  &1.0757   &6.02e-05	&7.62e-12	&6.99e-14		&52.93  &3.4949  &5.93e-05	&1.12e-11  &8.22e-14\\
		            &max	&57	   	&1.9995   &9.87e-05	&3.36e-11	&7.72e-14		&61	   	&5.9582  &9.99e-05	&3.42e-11  &9.98e-14\\
\hline
		            &min	&36	   	&1.3851   &1.13e-05	&1.03e-13	&7.58e-15    	&36	   	&4.3346  &1.65e-05	&9.48e-14  &8.53e-15\\
\textbf{PGST}		&mean	&41.03  &1.9658   &6.20e-05	&1.16e-11	&1.02e-14		&41.63  &5.6442  &6.17e-05	&6.58e-12  &1.03e-14\\
		            &max	&47	   	&3.4550   &9.92e-05	&3.87e-11	&1.49e-14		&46	   	&7.2430  &9.92e-05	&2.54e-11  &1.22e-14\\
\hline
		            &min	&47	   	&1.0511   &8.47e-06	&2.06e-13	&5.10e-14		&47	   	&3.3157  &1.51e-05	&2.69e-13  &6.60e-14\\
\textbf{Grad-retrac}&mean	&52.80  &1.5083   &6.40e-05	&1.18e-11	&5.67e-14		&53.67  &4.6731  &6.06e-05	&9.91e-12  &7.10e-14\\
		            &max	&69	   	&2.5606   &9.86e-05	&3.80e-11	&9.93e-14		&62	   	&6.6604  &9.72e-05	&3.62e-11  &9.43e-14\\
\hline
\hline
  & & \multicolumn{5}{|c||}{ Ex.5: $n = 800$, $p = 450$, $problem = 1$  }& \multicolumn{5}{|c|}{ Ex.6: $n = 1000$, $p = 500$, $problem = 1$  }\\
\hline
		            &min	&56	   	&24.126   &3.50e-05	&2.54e-12	&2.15e-14		&57	   	&43.799  &3.59e-05	&2.14e-12  &2.43e-14\\
\textbf{OptStiefel} &mean	&63.27  &33.316   &7.37e-05	&6.30e-11	&2.23e-14		&62.13  &59.386  &7.53e-05	&5.05e-11  &2.53e-14\\
		            &max	&70	   	&39.796   &9.86e-05	&1.36e-10	&2.31e-14		&74	   	&103.86  &9.71e-05	&1.30e-10  &2.65e-14\\
\hline
		            &min	&39	   	&23.808   &3.05e-05	&1.80e-12	&9.57e-15		&40	   	&41.973  &1.39e-05	&4.45e-13  &1.07e-14\\
\textbf{PGST}		&mean	&42.03  &31.916   &6.28e-05	&1.88e-11	&1.17e-14		&42.60  &57.541  &5.22e-05	&1.22e-11  &1.23e-14\\
		            &max	&44	   	&41.976   &9.80e-05	&5.97e-11	&1.40e-14		&44	   	&139.13  &9.16e-05	&3.41e-11  &1.47e-14\\
\hline
		            &min	&43	   	&21.228   &2.48e-05	&1.07e-12	&7.66e-14		&41	   	&32.012  &4.21e-05	&2.46e-12  &7.79e-14\\
\textbf{Grad-retrac}&mean	&48.73  &26.162   &6.63e-05	&3.41e-11	&8.77e-14		&45.50  &39.921  &7.14e-05	&3.98e-11  &9.09e-14\\
		            &max	&55	   	&31.251   &9.95e-05	&1.19e-10	&9.85e-14		&52	   	&58.320  &9.66e-05	&1.01e-10  &9.96e-14\\
\hline
\end{tabular}}
\end{table}

\begin{table}
\centering
\caption{\small{Performance of the methods for ill conditioned WOPP problems}}
\label{tab:4}
\resizebox{12cm}{!}{\begin{tabular}{|l|c|c|c|c|c|c||c|c|c|c|c|}
  \hline

  \textbf{Methods} &	    &\textbf{Nitr}    &   \textbf{Time}	&\textbf{NrmG}	&\textbf{Fval}	&\textbf{Feasi}	&\textbf{Nitr}	 &\textbf{Time}	&\textbf{NrmG}	&\textbf{Fval}	&\textbf{Feasi} \\
  \hline	
  \hline		
  & & \multicolumn{5}{|c||}{ Ex.7: $n = 100$, $p = 50$, $problem = 2$  }& \multicolumn{5}{|c|}{ Ex.8: $n = 500$, $p = 20$, $problem = 2$   }\\
  \hline
		            &min	&614	 &640	  &1.36e-04 &3.61e-11	&2.79e-15		&4841	&13.0517  &9.32e-04	&5.48e-09	&1.87e-15\\	
\textbf{OptStiefel}	&mean	&903.40	 &1.3854  &6.26e-04 &3.69e-09	&3.38e-15		&9827	&30.1954  &0.0989	&0.0288	    &4.96e-14\\
		            &max	&1348	 &5.4052  &4.90e-03 &8.42e-09	&3.93e-15		&15000	&58.5938  &2.6101	&0.6063	    &9.84e-14\\	
\hline
		            &min	&556	 &934	  &1.16e-04 &2.79e-11	&4.93e-15		&4519	&25.8166  &5.20e-04	&6.72e-09	&4.40e-15\\	
\textbf{PGST}		&mean	&857.17	 &3.3145  &6.89e-04 &2.33e-09	&6.97e-15		&6849.1 &46.7116  &0.0159	&1.11e-05	&6.67e-15\\	
		            &max	&1440	 &5.6725  &6.60e-03 &2.62e-08	&1.12e-14		&11815	&96.0478  &0.3035	&3.27e-04	&9.85e-15\\	
\hline		
		            &min	&632	 &653	  &1.40e-04	&3.91e-11	&1.03e-14		&5176	&13.0955  &5.84e-04	&1.38e-08	&7.14e-15\\	
\textbf{Grad-retrac}&mean	&929.53	 &1.8521  &3.99e-04	&2.53e-09	&3.02e-14		&9702.8 &28.5254  &0.0294	&0.0385	    &4.42e-14\\	
		            &max	&1925	 &3.6914  &1.70e-03	&5.76e-09	&9.65e-14		&15000	&61.9137  &0.5755	&0.6063	    &9.91e-14\\	
  \hline
  \hline
  & & \multicolumn{5}{|c||}{ Ex.9: $n = 100$, $p = 100$, $problem = 2$  }& \multicolumn{5}{|c|}{ Ex.10: $n = 200$, $p = 200$, $problem = 2$  }\\
  \hline
			        &min	&477    &1.5892	  &1.17e-04	&8.28e-11	&5.20e-15       &1088	 &13.841  &3.37e-04	&1.21e-10	&1.05e-14\\
\textbf{OptStiefel} &mean	&679.50	&2.5212	  &5.60e-04	&2.36e-09	&5.59e-15	    &1611.60 &20.469  &3.30e-03	&3.50e-08	&1.10e-14\\
		            &max	&1102	&6.5418	  &1.80e-03	&6.40e-09	&6.19e-15		&2961	 &37.932  &2.03e-02	&3.70e-07	&1.16e-14\\
  \hline
		            &min	&543	&5.4842	  &6.08e-05	&5.82e-13	&6.22e-15	    &1292	 &45	  &2.21e-04	&1.04e-11	&7.90e-15\\
\textbf{PGST}		&mean	&859.73	&8.8871	  &6.16e-04	&1.14e-09	&8.76e-15	    &2138.00 &77.01	  &3.00e-03	&1.45e-08	&1.02e-14\\	
		            &max	&1460	&16.469	  &2.70e-03	&6.01e-09	&1.18e-14	    &3847	 &138	  &2.28e-02	&1.92e-07	&1.34e-14\\
\hline
		            &min	&480	&2.4762	  &1.40e-04	&5.25e-11	&2.40e-14		&900	 &15.509  &5.75e-04	&3.80e-09	&2.77e-14\\
\textbf{Grad-retrac}&mean	&665.87	&3.2475	  &1.60e-03	&2.22e-09	&4.10e-14		&1658.80 &25.174  &2.10e-03	&1.41e-08	&6.51e-14\\	
		            &max	&954	&4.7347	  &2.39e-02	&9.53e-09	&8.67e-14	    &3036	 &43.884  &9.20e-03	&3.51e-08	&9.95e-14\\
  \hline
  \hline
  & & \multicolumn{5}{|c||}{ Ex.11: $n = 200$, $p = 100$, $problem = 2$  }& \multicolumn{5}{|c|}{ Ex.12: $n = 200$, $p = 150$, $problem = 2$  }\\
  \hline
		            &min	&1537	&9.7212	  &4.35e-04	&2.92e-09	&5.64e-15		&1450	 &14.489  &3.30e-04	&1.25e-09	&8.06e-15\\
\textbf{OptStiefel} &mean	&2291.1	&14.680	  &1.60e-03	&2.95e-08	&6.37e-15		&2158.40 &21.815  &1.90e-03	&1.64e-08	&8.64e-15\\	
		            &max	&4042	&27.296	  &8.70e-03	&2.35e-07	&7.02e-15	    &4829	 &48.014  &9.60e-03	&3.91e-08	&9.44e-15\\
\hline
		            &min	&1057	&13.740	  &2.54e-04	&1.26e-09	&6.21e-15	    &1164	 &26.374  &2.01e-04	&1.11e-10	&7.11e-15\\
\textbf{PGST}		&mean	&2152.4 &27.127	  &1.40e-03	&1.16e-08	&8.45e-15	    &2125.50 &49.825  &2.00e-03	&8.35e-09	&8.33e-15\\
		            &max	&3180	&39.891	  &3.50e-03	&3.66e-08	&1.29e-14	    &3337	 &77.428  &7.30e-03	&3.26e-08	&1.09e-14\\
\hline
		            &min	&1454	&9.4400	  &3.93e-04 &4.34e-09	&1.81e-14	    &1263	 &17.645  &3.11e-04	&5.97e-10	&2.22e-14\\
\textbf{Grad-retrac}&mean	&2743.4 &16.738   &2.40e-03	&3.45e-08	&5.11e-14		&2268.10 &27.095  &2.00e-03	&2.35e-08	&4.49e-14\\
		            &max	&5344	&32.535	  &3.16e-02	&3.53e-07	&9.82e-14	    &3422	 &38.263  &1.30e-02	&1.25e-07	&9.56e-14\\
 \hline
\end{tabular}}
\end{table}

\begin{table}
\centering
\caption{\small{Performance of the methods for ill conditioned WOPP problems}}
\label{tab:5}
\resizebox{12cm}{!}{\begin{tabular}{|l|c|c|c|c|c|c||c|c|c|c|c|}
  \hline
  \textbf{Methods} &	    &\textbf{Nitr}    &   \textbf{Time}	&\textbf{NrmG}	&\textbf{Fval}	&\textbf{Feasi}	&\textbf{Nitr}	 &\textbf{Time}	&\textbf{NrmG}	&\textbf{Fval}	&\textbf{Feasi} \\
  \hline
  \hline
  & & \multicolumn{5}{|c||}{ Ex.13: $n = 500$, $p = 20$, $problem = 3$  }& \multicolumn{5}{|c|}{ Ex.14: $n = 1000$, $p = 10$, $problem = 3$   }\\
  \hline
  		            &min	&1207    &5.7719	&1.32e-04	&1.57e-09	&3.54e-15		&1206	 &15.846  &1.69e-04  &6.67e-09	&7.80e-15\\
\textbf{OptStiefel} &mean	&1888.00 &10.739	&7.08e-04	&2.72e-08	&3.53e-14		&2299.90 &29.248  &9.30e-04  &8.61e-08	&1.89e-14\\
		            &max	&3135	 &19.596	&3.00e-03	&1.24e-07	&7.58e-14	    &3858	 &48.623  &4.30e-03  &8.44e-07	&6.24e-14\\	
\hline
		            &min	&986	 &9.8174	&5.33e-05	&1.53e-11	&3.31e-15	    &1288	 &25.262  &1.43e-04  &1.25e-09	&2.33e-15\\
\textbf{PGST}		&mean	&1872.10 &17.895	&6.48e-04	&2.86e-08	&6.67e-15	    &2112.40 &39.082  &7.39e-04  &3.95e-08	&4.16e-15\\
		            &max	&3140	 &30.382	&4.70e-03	&4.22e-07	&1.03e-14		&3547	 &64.554  &3.00e-03  &2.82e-07	&6.44e-15\\
\hline
		            &min	&1208	 &5.3518	&9.00e-05	&1.19e-09	&6.31e-15		&1020	 &13.023  &1.97e-04  &6.49e-09	&3.37e-15\\
\textbf{Grad-retrac}&mean	&1882.20 &10.263	&4.73e-04	&2.33e-08	&3.00e-14		&2382.40 &29.756  &1.30e-03  &9.70e-08	&2.24e-14\\	
		            &max	&2897	 &19.609	&2.40e-03	&4.68e-08	&9.65e-14		&4268	 &52.634  &1.21e-02  &9.39e-07	&8.15e-14\\	
\hline
\hline
  & & \multicolumn{5}{|c||}{ Ex.15: $n = 100$, $p = 100$, $problem = 3$  }& \multicolumn{5}{|c|}{ Ex.16: $n = 200$, $p = 200$, $problem = 3$  }\\
\hline
		            &min	&481	 &1.4359	&1.05e-04	&6.33e-11	&5.25e-15	    &568	 &7.4377  &1.69e-04  &3.21e-11  &1.06e-14\\
\textbf{OptStiefel} &mean	&706.77	 &2.1294	&8.52e-04	&2.21e-09	&5.55e-15		&876.33	 &11.355  &8.42e-04  &3.92e-09	&1.10e-14\\	
		            &max	&1148	 &3.3911	&3.30e-03	&7.76e-09	&5.88e-15		&1336	 &16.936  &3.50e-03  &1.27e-08	&1.17e-14\\
\hline
		            &min	&440	 &3.6841	&6.46e-05	&8.50e-13	&6.79e-15		&542	 &19.691  &7.02e-05  &9.63e-13	&8.57e-15\\
\textbf{PGST}		&mean	&858.13	 &7.6421	&5.37e-04	&8.98e-10	&8.93e-15		&1042.10 &36.663  &1.30e-03  &5.87e-09	&1.04e-14\\
		            &max	&1542	 &14.094	&1.40e-03	&7.55e-09	&1.31e-14	    &1410	 &50.839  &9.50e-03  &1.05e-07	&1.33e-14\\
\hline
		            &min	&426	 &1.9723	&2.11e-04	&1.98e-10	&2.23e-14		&512	 &9.5307  &1.35e-04  &2.82e-11	&2.81e-14\\
\textbf{Grad-retrac}&mean	&680.83	 &2.8749	&7.41e-04	&2.59e-09	&3.40e-14		&848.77	 &14.911  &1.00e-03  &4.69e-09	&5.37e-14\\
		            &max	&1098	 &4.5067	&3.40e-03	&9.91e-09	&7.66e-14		&1290	 &21.070  &8.80e-03  &2.58e-08	&9.15e-14\\
\hline
\hline
  & & \multicolumn{5}{|c||}{ Ex.17: $n = 200$, $p = 100$, $problem = 3$  }& \multicolumn{5}{|c|}{ Ex.18: $n = 200$, $p = 150$, $problem = 3$  }\\
\hline

		            &min	&696	 &5.1556	&1.48e-04	&1.52e-10	&5.19e-15		&632	 &6.2744  &1.24e-04  &9.49e-11	&7.99e-15\\
\textbf{OptStiefel} &mean	&1176.80 &10.889	&1.40e-03	&7.29e-09	&5.63e-15	    &946.20	 &9.4612  &7.73e-04  &4.39e-09	&8.62e-15\\	
		            &max	&2276	 &19.284	&1.33e-02	&4.92e-08	&6.07e-15	    &1394	 &13.877  &2.70e-03  &1.22e-08	&9.17e-15\\
\hline
		            &min	&520	 &10.374	&6.55e-05	&8.36e-13	&6.04e-15   	&630	 &13.926  &6.84e-05  &1.53e-12	&7.60e-15\\
\textbf{PGST}	    &mean	&1011.20 &18.462	&1.10e-03	&1.38e-08	&8.16e-15	    &967.80	 &21.567  &4.85e-04  &1.06e-09	&8.74e-15\\	
		            &max	&1481	 &34.484	&1.29e-02	&3.58e-07	&1.20e-14		&1509	 &34.140  &2.00e-03	 &6.82e-09	&1.09e-14\\
\hline
		            &min	&813	 &7.8483	&7.77e-05	&1.48e-10	&1.65e-14	    &692	 &9.4317  &7.39e-05	 &2.36e-12	&2.07e-14\\
\textbf{Grad-retrac}&mean	&1231.10 &11.9001	&1.30e-03	&8.66e-09	&4.29e-14       &1120.50 &15.026  &1.20e-03	 &8.79e-09	&4.19e-14\\
		            &max	&1887	 &21.735	&6.80e-03	&7.18e-08	&9.27e-14	    &1908	 &26.238  &8.60e-03	 &8.65e-08	&9.66e-14\\
  \hline
\end{tabular}}
\end{table}

In Figure \ref{fig:2}, we give the convergence history of each method in the WOPP's problems shown in the tables \ref{tab:3}-\ref{tab:5}. More specifically, we depict the average assessments of objective function as well as the average gradient norm for Ex.5, Ex.9 and Ex.15 experiments. In these charts, we observe that in well conditioned problems (see Figure \ref{fig:2}(a)-(b)) the method \textbf{PGST} converge faster than the other, but in the presence of ill conditioning, our method is faster. In addition, in all charts, \textbf{OptStiefel} an our method showed similar performance.

\begin{figure}
  \centering
  \begin{center}
  \subfigure[Objective Function, Table \ref{tab:3}, Ex.5 ]{\includegraphics[width=5cm]{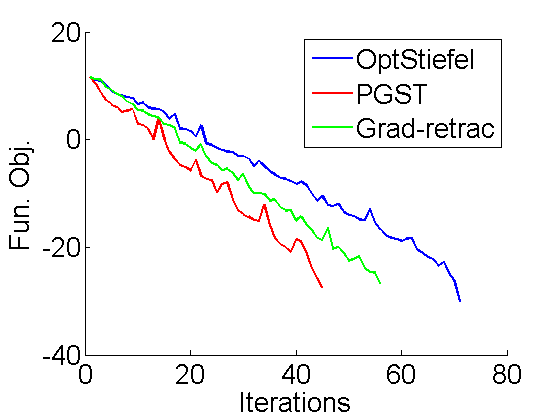}}  \hspace{1.5cm}
  \subfigure[Gradient norm, Table \ref{tab:3}, Ex.5  ]{\includegraphics[width=5cm]{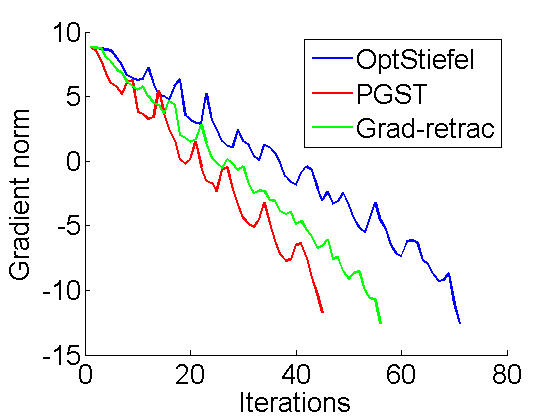}}

  \subfigure[Objective Function, Table \ref{tab:4}, Ex.9 ]{\includegraphics[width=5cm]{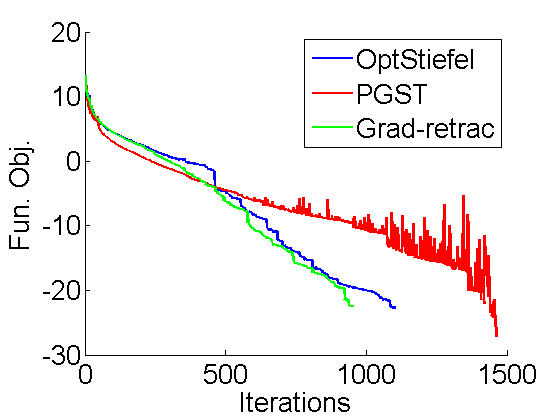}}  \hspace{1.5cm}
  \subfigure[Gradient norm, Table \ref{tab:4}, Ex.9  ]{\includegraphics[width=5cm]{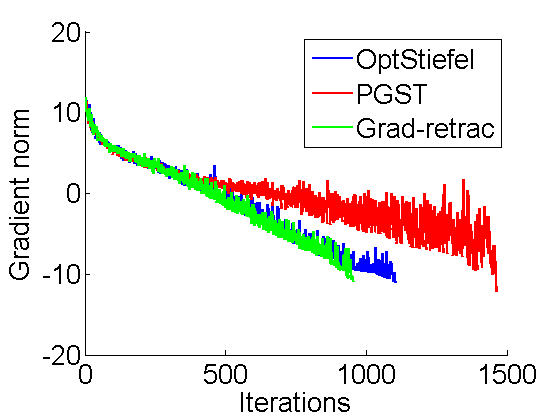}}

  \subfigure[Objective Function, Table \ref{tab:5}, Ex.15 ]{\includegraphics[width=5cm]{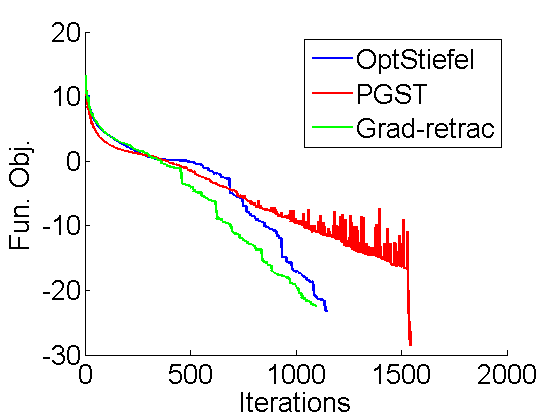}}  \hspace{1.5cm}
  \subfigure[Gradient norm, Table \ref{tab:5}, Ex.15  ]{\includegraphics[width=5cm]{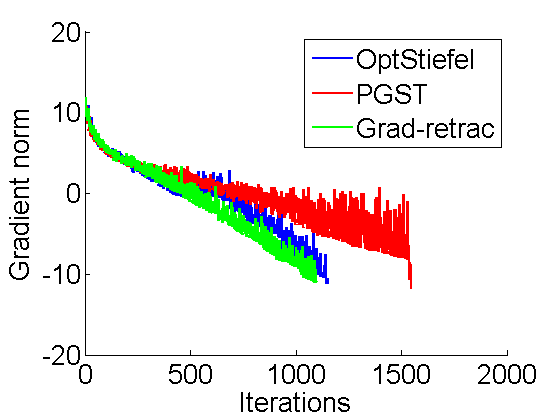}}
  \end{center}
  \caption{Graphics assessments of the average objective function and the average gradient norm from 30 simulations. The y-axis is on a
  logarithmic scale.}
  \label{fig:2}
\end{figure}

\subsection{Total energy minimization}
\label{subsec:8}
For next experiments we consider the following total energy minimization problem:
\begin{equation}
  \min_{X\in \mathbb{R}^{n\times k}} E_{total}(X) = \frac{1}{2}Tr[X^{\top}LX] + \frac{\mu}{4}\rho(X)^{\top}L^{\dag}\rho(X) \quad s.t. \quad X^{\top}X = I_k \label{TE_problem} \nonumber
\end{equation}
where $L$ is a discrete Laplacian operator, $\mu>0$ is a given constant, $\rho(X) := diag(XX^{\top})$ is the vector containing the diagonal elements of the matrix $XX^{\top}$ and $L^{\dag}$ is the Moore-Penrose generalized inverse of matrix $L$. The total energy minimization problem (\ref{TE_problem}) is a simplified version of the HartreeFock (HF) total energy minimization problem and the Kohn-Sham (KS) total energy minimization problem in electronic structure calculations (see for instance \cite{martin2004electronic,ostlund1996modern,yang2006constrained,yang2007trust}). The first order necessary conditions for the total energy minimization problem (\ref{TE_problem}) are given by:
\begin{eqnarray}
  H(X)X - X\Lambda &=& 0 \nonumber \\
  X^{\top}X &=& I_k, \nonumber
\end{eqnarray}
where the diagonal matrix $\Lambda$ contains the $k$ smallest eigenvalues of the symmetric matrix $H(X) := L + \mu Diag(L^{\dag}\rho(X))$. The symbol $Diag(x)$ is a diagonal matrix with a vector $x$ on its diagonal.\\

For examples 6.1-6.4 below taken from \cite{ZhaoBai}, we repeat our experiments over 100 different starting points, moreover, we use a tolerance of $\epsilon =$ 1e-4 and a maximum number of iterations $K = 1000$. To show the effectiveness of our method over the problem (\ref{TE_problem}), we report the numerical results for Examples 6.1-6.4 with different choices of $n$, $k$, and $\mu$, and we compare our method with the Steepest Descent method (\textbf{Steep-Dest}), Trust-Region method (\textbf{Trust-Reg}) and Conjugate Gradient method (\textbf{Conj-Grad}) from ``\emph{manopt}'' tool­box\footnote{The tool-box manopt is available in http://www.manopt.org/}.\\

\underline{Example 6.1 \cite{ZhaoBai}}: We consider the nonlinear eigenvalue problem for different
choices of $n$; $k$; $\mu$: (a) $n = 2$; $k = 1$; $\mu = 3$; (b) $n = 10$; $k = 2$; $\mu = 0.6$; (c) $n = 100$; $k = 10$; $\mu = 0.005$; (d) $n = 100$; $k = 4$; $\mu = 0.001$. \\

\underline{Example 6.2 \cite{ZhaoBai}}: We consider the nonlinear eigenvalue problem for different
choices of $n$; $k$; $\mu$: (a) $n = 2$; $k = 1$; $\mu = 9$; (b) $n = 10$; $k = 2$; $\mu = 3$; (c) $n = 100$; $k = 10$; $\mu = 1$; (d) $n = 100$; $k = 4$; $\mu = 2$.\\

\underline{Example 6.3 \cite{ZhaoBai}}: We consider the nonlinear eigenvalue problem for $k = 10$; $\mu = 1$,
and varying $n = 200, 400, 800, 1000$.\\

\underline{Example 6.4 \cite{ZhaoBai}}: We consider the nonlinear eigenvalue problem for $n = 100$ and
$k = 20$ and varying $\mu$.\\

In all these experiments, the $L$ matrix in the problem (\ref{TE_problem}) is constructed as the one-dimensional discrete
Laplacian with $2$ on the diagonal and $-1$ on the sub- and sup-diagonals. Furthermore, for all these experiments we select: $\alpha = 0.7$, and $\beta = 0.3$ in the descent direction of our method. \\

The results for the Example 6.1 and Example 6.2 are shown in Tables \ref{tab:6}-\ref{tab:7}. We see from these tables that our method (\textbf{Grad-retrac}) is more efficient than the other methods from ``\emph{manopt}'' tool­box in terms of CPU time. Table \ref{tab:8}
gives numerical results for Example 6.3. We see from Table \ref{tab:8} that in almost all cases, our method is more efficient than the other
in terms of CPU time. Only in experiments Ex.11 and Ex.12 (see Table \ref{tab:8}), the Conjugate Gradient method from ``\emph{manopt}'' library gets better performance than our method slightly. In addition, all algorithms take longer to converge as it increases the value of $n$, nevertheless, when the problems are larger size, the increase in the convergence time of the \textbf{Conj-Grad} and \textbf{Grad-retrac} methods is much less than the of the other methods.\\


\begin{table}
\centering
\caption{\small{Numerical results for Example 6.1}}
\label{tab:6}
\resizebox{12cm}{!}{\begin{tabular}{|l|c|c|c|c|c||c|c|c|c|c|}
  \hline
  & \multicolumn{5}{|c||}{ Ex.1:  $(n,k,\mu) =  (2,1,3)$  }& \multicolumn{5}{|c|}{ Ex.2:  $(n,k,\mu) =    (10,2,0.6) $  }\\
  \hline
  \hline
  \textbf{Methods} &	    \textbf{Nitr}	 &   \textbf{Time}	&\textbf{NrmG}	&\textbf{Fval}	&\textbf{Feasi}	&\textbf{Nitr}	&\textbf{Time}	&\textbf{NrmG}	&\textbf{Fval}	&\textbf{Feasi} \\
  \hline
  \textbf{Trust-Reg}        &4.48	&0.0141	&1.19e-04	&0.8750	&1.95e-16			&7.79	&0.0264	&1.51e-04	&0.8495	&5.27e-16\\
  \textbf{Steep-Dest}       &6.52	&0.0160	&4.22e-04	&0.8750	&2.00e-16			&23.54	&0.0519	&5.51e-04	&0.8495	&5.12e-16\\
  \textbf{Conj-Grad}        &5.62	&0.0126	&7.40e-03	&0.8752	&1.94e-16			&17.08	&0.0309	&4.60e-03	&0.8513	&5.31e-16\\				
  \textbf{Grad-retrac}      &5.72	&0.0032	&1.56e-05	&0.8750	&2.96e-15			&22.11	&0.0087	&6.47e-05	&0.8495	&1.21e-14\\
  \hline
  \hline
  & \multicolumn{5}{|c||}{ Ex.3:  $(n,k,\mu) =   (100,10,0.005)$  }& \multicolumn{5}{|c|}{ Ex.4:  $(n,k,\mu) =   (100,4,0.001) $  }\\
  \hline
  \textbf{Trust-Reg}        &13.00	&0.1524	&1.81e-04	&1.0547	&2.58e-15			&11.80	&0.1632	&3.05e-04	&5.02e-02	&1.37e-15\\
  \textbf{Steep-Dest}       &535.71	&2.1547	&3.58e-04	&1.0548	&1.99e-15			&930.94	&3.1733	&2.46e-04	&5.02e-02	&1.32e-15\\
  \textbf{Conj-Grad}        &69.32	&0.1897	&8.14e-04	&1.0548	&2.72e-15			&80.46	&0.1901	&8.92e-04	&5.03e-02	&1.45e-15\\			
  \textbf{Grad-retrac}      &105.29	&0.0782	&7.46e-05	&1.0547	&2.52e-14			&129.09	&0.0711	&8.03e-05	&5.02e-02	&2.55e-14\\
  \hline
\end{tabular}}
\end{table}

\begin{table}
\centering
\caption{\small{Numerical results for Example 6.2}}
\label{tab:7}
\resizebox{12cm}{!}{\begin{tabular}{|l|c|c|c|c|c||c|c|c|c|c|}
  \hline
  & \multicolumn{5}{|c||}{ Ex.5:  $(n,k,\mu) =  (2,1,9)$  }& \multicolumn{5}{|c|}{ Ex.6:  $(n,k,\mu) =     (10,2,3) $  }\\
  \hline
  \hline
  \textbf{Methods} &	    \textbf{Nitr}	 &   \textbf{Time}	&\textbf{NrmG}	&\textbf{Fval}	&\textbf{Feasi}	&\textbf{Nitr}	&\textbf{Time}	&\textbf{NrmG}	&\textbf{Fval}	&\textbf{Feasi} \\
  \hline
  \textbf{Trust-Reg}    &3.53	&0.0090	&1.07e-04	&1.9350	&2.00e-16			&8.63	&0.0260	&2.33e-04	&2.5046	&5.78e-16\\
  \textbf{Steep-Desct}  &7.15	&0.0172	&4.22e-04	&1.6250	&2.09e-16			&17.76	&0.0392	&5.41e-04	&2.5046	&5.16e-16\\
  \textbf{Conj-Grad}    &4.75	&0.0106	&5.23e-02	&1.8380	&2.25e-16			&15.69	&0.0284	&7.12e-04	&2.5046	&5.50e-16\\			
  \textbf{Grad-retrac}  &5.47	&0.0028	&1.37e-05	&1.9150	&3.24e-15			&19.35	&0.0077	&5.92e-05	&2.5046	&1.92e-14\\

  \hline
  \hline
  & \multicolumn{5}{|c||}{ Ex.7:  $(n,k,\mu) =   (100,10,1)$  }& \multicolumn{5}{|c|}{ Ex.8:  $(n,k,\mu) =   (100,4,2) $  }\\
  \hline
  \textbf{Trust-Reg}    &19.77	&0.1622	&5.53e-04	&35.7086	&2.76e-15		&12.10	&0.0601	&2.66e-04	&7.7005	&1.42e-15\\
  \textbf{Steep-Desct}  &149.89	&0.5982	&4.33e-04	&35.7086	&1.96e-15	    &42.33	&0.1260	&4.65e-04	&7.7005	&1.48e-15\\
  \textbf{Conj-Grad}    &44.03	&0.1134	&8.10e-04	&35.7086	&3.04e-15		&25.76	&0.0558	&7.00e-04	&7.7005	&1.48e-15\\			
  \textbf{Grad-retrac}  &65.33	&0.0464	&7.47e-05	&35.7086	&2.74e-14		&36.77	&0.0200	&6.58e-05	&7.7005	&1.85e-14\\
  \hline
\end{tabular}}
\end{table}
Table \ref{tab:9} lists numerical results for Example 6.4. In this table, we note that our method gets its best performance when it used the smaller size of $\mu$, we also observe that our method is more efficient than the other in all of the experiments list in this table.

\begin{table}
\centering
\caption{\small{Numerical results for Example 6.3}}
\label{tab:8}
\resizebox{12cm}{!}{\begin{tabular}{|l|c|c|c|c|c||c|c|c|c|c|}
  \hline
  & \multicolumn{5}{|c||}{ Ex.9:  $(n,k,\mu) =  (200,10,1)$  }& \multicolumn{5}{|c|}{ Ex.10:  $(n,k,\mu) =     (400,10,1) $  }\\
  \hline
  \hline
  \textbf{Methods} &	    \textbf{Nitr}	 &   \textbf{Time}	&\textbf{NrmG}	&\textbf{Fval}	&\textbf{Feasi}	&\textbf{Nitr}	&\textbf{Time}	&\textbf{NrmG}	&\textbf{Fval}	&\textbf{Feasi} \\
  \hline
  \textbf{Trust-Reg}    &20.13	&0.2146	&5.28e-04	&35.7086	&2.90e-15			&21.29	&0.6052	&5.42e-04	&35.7086	&3.05e-15\\
  \textbf{Steep-Dest}   &156.38	&0.7624	&3.70e-04	&35.7086	&2.07e-15			&155.39	&1.1328	&3.92e-04	&35.7086	&2.19e-15\\
  \textbf{Conj-Grad	}   &45.01	&0.1362	&7.85e-04	&35.7086	&3.07e-15			&46.74	&0.2308	&7.92e-04	&35.7086	&3.30e-15\\			
  \textbf{Grad-retrac}  &65.38	&0.0654	&7.48e-05	&35.7086	&2.23e-14			&67.27	&0.1644	&7.17e-05	&35.7086	&2.77e-14\\
  \hline
  \hline
  & \multicolumn{5}{|c||}{ Ex.11:  $(n,k,\mu) =   (800,10,1)$  }& \multicolumn{5}{|c|}{ Ex.12:  $(n,k,\mu) =   (1000,10,1)  $  }\\
  \hline
  \textbf{Trust-Reg}    &21.23	&2.3096	&5.50e-04	&35.7086	&3.21e-15			&21.78	&2.9471	&5.67e-04	&35.7086	&3.40e-15\\
  \textbf{Steep-Dest}   &158.52	&2.6007	&3.71e-04	&35.7086	&2.78e-15			&157.96	&2.7528	&4.17e-04	&35.7086	&3.48e-15\\
  \textbf{Conj-Grad	}   &47.52	&0.5641	&8.14e-04	&35.7086	&3.77e-15			&47.53	&0.6426	&8.27e-04	&35.7086	&3.94e-15\\				
  \textbf{Grad-retrac}  &67.13	&0.5777	&7.55e-05	&35.7086	&3.41e-14			&68.16	&0.7171	&7.52e-05	&35.7086	&2.94e-14\\
  \hline
\end{tabular}}
\end{table}

\begin{table}
\centering
\caption{\small{Numerical results for Example 6.4}}
\label{tab:9}
\resizebox{12cm}{!}{\begin{tabular}{|l|c|c|c|c|c||c|c|c|c|c|}
  \hline
  & \multicolumn{5}{|c||}{ Ex.13:  $(n,k,\mu) =   (100,20,0.0001) $  }& \multicolumn{5}{|c|}{ Ex.14:  $(n,k,\mu) =     (100,20,0.001) $  }\\
  \hline
  \hline
  \textbf{Methods} &	    \textbf{Nitr}	 &   \textbf{Time}	&\textbf{NrmG}	&\textbf{Fval}	&\textbf{Feasi}	&\textbf{Nitr}	&\textbf{Time}	&\textbf{NrmG}	&\textbf{Fval}	&\textbf{Feasi} \\
  \hline
  \textbf{Trust-Reg}    &11.27	&0.1060	&2.56e-04	&1.4484	&4.34e-15			&11.17	&0.1061	&3.09e-04	&2.2066	&4.39e-15\\
  \textbf{Steep-Dest}   &253.37	&1.3326	&3.77e-04	&1.4484	&2.83e-15			&238.17	&1.2122	&4.09e-04	&2.2066	&2.60e-15\\
  \textbf{Conj-Grad}    &48.24	&0.1554	&8.46e-04	&1.4484	&5.44e-15			&48.80	&0.1585	&8.72e-04	&2.2066	&5.37e-15\\					
  \textbf{Grad-retrac}  &74.24	&0.0776	&7.57e-05	&1.4484	&3.14e-14			&73.26	&0.0784	&7.49e-05	&2.2066	&2.56e-14\\
  \hline
  \hline
  & \multicolumn{5}{|c||}{ Ex.15:  $(n,k,\mu) =    (100,20,0.01)$  }& \multicolumn{5}{|c|}{ Ex.16:  $(n,k,\mu) =    (100,20,0.1)  $  }\\
  \hline
  \textbf{Trust-Reg}    &15.34	&0.2084	&3.87e-04	&7.8706	&4.19e-15			&21.33	&0.3239	&6.15e-04	&33.7574	&5.01e-15\\
  \textbf{Steep-Dest}   &391.23	&2.2931	&2.62e-04	&7.8706	&2.60e-15			&448.39	&3.7625	&2.43e-04	&33.7574	&2.54e-15\\
  \textbf{Conj-Grad}    &60.44	&0.2243	&8.42e-04	&7.8706	&5.62e-15			&64.84	&0.3026	&8.66e-04	&33.7574	&5.85e-15\\		
  \textbf{Grad-retrac}  &87.43	&0.1070	&7.81e-05	&7.8706	&4.09e-14			&99.48	&0.1463	&8.01e-05	&33.7574	&4.02e-14\\
  \hline
  \hline
  & \multicolumn{5}{|c||}{ Ex.17:  $(n,k,\mu) =    (100,20,1)$  }& \multicolumn{5}{|c|}{ Ex.18:  $(n,k,\mu) =    (100,20,20)  $  }\\
  \hline
  \textbf{Trust-Reg}    &25.30	&0.2863	&5.99e-04	&2.11e+02	&5.48e-15			&34.15	 &0.3929	&5.17e-04	&3.87e+03	&5.84e-15\\
  \textbf{Steep-Dest}   &555.11	&3.6971	&3.04e-04	&2.11e+02	&2.67e-15			&1797.70 &18.3231	&3.90e-04	&3.87e+03	&2.73e-15\\
  \textbf{Conj-Grad}    &81.45	&0.2726	&8.57e-04	&2.11e+02	&5.46e-15			&153.16	 &0.5131	&8.63e-04	&3.87e+03	&3.50e-15\\			
  \textbf{Grad-retrac}  &116.49	&0.1222	&1.35e-04	&2.11e+02	&2.62e-14			&160.62	 &0.1661	&2.70e-03	&3.87e+03	&2.55e-14\\
  \hline
  \hline
  & \multicolumn{5}{|c||}{ Ex.19:  $(n,k,\mu) =   (100,20,40)$  }& \multicolumn{5}{|c|}{ Ex.20:  $(n,k,\mu) =   (100,20,80)  $  }\\
  \hline
  \textbf{Trust-Reg}    &37.46	 &0.5130	 &5.90e-04	&7.72e+03	&6.14e-15			&41.35	    &0.5403	 &6.63e-04	&1.54e+04	&6.25e-15\\
  \textbf{Steep-Dest}   &2162.70 &21.7257    &3.67e-04	&7.72e+03	&2.69e-15			&3585.60	&37.5829 &3.61e-04	&1.54e+04	&2.55e-15\\
  \textbf{Conj-Grad}    &179.83	 &0.7251	 &8.63e-04	&7.72e+03	&3.17e-15			&212.05	    &0.8118	 &8.55e-04	&1.54e+04	&2.97e-15\\		
  \textbf{Grad-retrac}  &173.46	 &0.2159	 &5.60e-03	&7.72e+03	&2.77e-14			&192.18	    &0.2208	 &1.25e-02	&1.54e+04	&2.76e-14\\
  \hline
\end{tabular}}
\end{table}

\subsection{Linear eigenvalue problem}
\label{subsec:9}
Given a symmetric matrix $A\in \mathbb{R}^{n\times n}$  and let $\lambda_1\geq\ldots\geq \lambda_n$ be the eigenvalues of $A$. The
$p$-largest eigenvalue problem can be formulated as:
\begin{equation}
  \sum_{i=1}^{p}\lambda_i:= \max_{X\in \mathbb{R}^{n\times p}} Tr[X^{\top}AX] \quad s.t. \quad X^{\top}X = I \nonumber
\end{equation}

In this subsection, we compared Algorithm 2 with the \textbf{Sgmin} algorithm proposed in \cite{Edelman}\footnote{The Sgmin solver is available in http://web.mit.edu/$\sim$ripper/www/sgmin.html} and the \textbf{OptStiefel} proposed in \cite{WenYin}. In this case, $\nabla\mathcal{F}_1=\nabla\mathcal{F}_2$, so, our direction coincide with Manton's direction. In all experiments presented in this subsection, we generate the matrix $A$ as follows: $A = \overline{A}^{\top}\overline{A}$, where $\overline{A}\in \mathbb{R}^{n\times n}$ is a matrix whose elements are sampled from the standard Gaussian distribution. The tables \ref{tab:10}-\ref{tab:11} shows the average (mean) of the results achieved  for each value to compare from 100 simulations, in addition, we used 1000 by the maximum number of iterations and $\epsilon =$ 1e-5 as tolerance for each algorithm. For all methods to compare, we use the four stop criteria presented in subsection \ref{subsec:6} In this tables, \emph{Error} denotes the relative error between the objective values given by \emph{eigs} function of Matlab and the objective values obtain by each algorithm, i.e.
\begin{equation}
  Error := \frac{| \sum_{i=1}^{p} \lambda_i^{eig} - Tr[X_{est}^{\top}AX_{est}] |}{|Tr[X_{est}^{\top}AX_{est}]|} \nonumber
\end{equation}
where $\lambda_i^{eig}$ is the $i$-largest eigenvalue of $A$ calculated using the \emph{eig} function of Matlab, and $X_{est}$ denotes the estimated local optima for each algorithm.\\

The results corresponding to varying $p$ but fixed $n = 1000$ are presented in Table \ref{tab:10}, from these results we can observe that \textbf{OptStiefel} and our \textbf{Grad-retrac} are much more efficient methods that \textbf{Sgmin}. Moreover, we observe that \textbf{OptStiefel} and \textbf{Grad-retrac} show almost the same performance, we also see that when $p$ grows, our method converges faster \textbf{OptStiefel} in terms of CPU time (see Table \ref{tab:10} with $p=100$ and $p=200$). The second test compares the algorithms to a fixed value of $p$ and varying $n$, the numerical results of this test are presented in Table \ref{tab:11}, in this table we note that \textbf{OptStiefel} and \textbf{Grad-retrac} algorithms showed similar performance, while the method \textbf{Sgmin} shows poor performance.

\begin{table}[H]
\centering
\caption{\small{Eigenvalues on randomly generated dense matrices for fixed n = 1000 and varying p.}}
\label{tab:10}
\resizebox{8.5cm}{!}{\begin{tabular}{|l|c|c|c|c|c|c|}
  \hline
  \textbf{p} & \textbf{1} & \textbf{5} & \textbf{10} & \textbf{50} & \textbf{100} & \textbf{200} \\
  \hline
  \multicolumn{7}{|c|}{ \textbf{Sgmin}  }\\
  \hline
  \textbf{Nitr}  & 	37.91	  &61.19	 &77.33	    &111.39	    &124.65	    &176.82  \\
  \textbf{Time}  &	22.9206	  &54.318	 &79.3343	&173.187	&306.8699	&1128.10 \\
  \textbf{Nfe}   &	69.39	  &95.67	 &114.03	&153.91	    &169.93	    &226.69  \\
  \textbf{Fval}  &	3.97e+03  &1.94e+04	 &3.81e+04	&1.72e+05	&3.14e+05	&5.36e+05\\
  \textbf{NrmG}  &	9.67e-04  &2.40e-03	 &3.50e-03	&7.70e-03	&1.04e-02	&1.15e-02\\
  \textbf{Feasi} &	9.88e-17  &4.61e-16	 &7.20e-16	&2.55e-15	&4.33e-15	&7.74e-15\\
  \textbf{Error} &	9.63e-13  &2.09e-12	 &3.73e-12	&6.98e-12	&9.21e-12	&1.07e-11\\
  \hline
  \multicolumn{7}{|c|}{ \textbf{OptStiefel}  }\\
  \hline
  \textbf{Nitr}	 &  102.7	  &125.86	 &142.69	&185.84	    &213.33	    &439.94\\
  \textbf{Time}	 &  0.1595	  &0.4831	 &0.75	    &2.2888	    &5.4833	    &34.9373\\
  \textbf{Nfe}	 &  106.33	  &130.48	 &148.53	&194.1	    &223.61	    &462.64\\
  \textbf{Fval}	 &  3.97e+03  &1.94e+04	 &3.81e+04	&1.72e+05	&3.14e+05	&5.36e+05\\
  \textbf{NrmG}	 &  5.73e-04  &1.40e-03	 &1.80e-03	&3.90e-03	&4.10e-03	&1.12e-02\\
  \textbf{Feasi} &  2.12e-15  &7.07e-17	 &3.01e-15	&8.31e-15	&1.20e-14	&1.76e-14\\
  \textbf{Error} &  3.98e-13  &8.40e-13	 &1.23e-12	&2.03e-12	&2.47e-12	&3.36e-12\\
  \hline
  \multicolumn{7}{|c|}{ \textbf{Grad-retrac}  }\\
  \hline
  \textbf{Nitr}	 &  100.16	  &126.39	 &141.73	&180.38	    &213.31	    &257.51\\
  \textbf{Time}	 &  0.159	  &0.5138	 &0.7661	&2.413	    &5.4744	    &20.509\\
  \textbf{Nfe}	 &  102.92	  &130.26	 &146.85	&187.22	    &223.75	    &271.49\\
  \textbf{Fval}	 &  3.97e+03  &1.94e+04	 &3.81e+04	&1.72e+05	&3.14e+05	&5.36e+05\\
  \textbf{NrmG}	 &  4.95e-04  &1.50e-03	 &2.50e-03	&3.80e-03	&4.50e-03	&4.80e-03\\
  \textbf{Feasi} &  8.66e-15  &1.98e-14	 &2.07e-14	&3.37e-14	&3.78e-14	&4.66e-14\\
  \textbf{Error} &  5.04e-13  &7.46e-13	 &1.30e-12	&2.36e-12	&2.22e-12	&3.96e-12\\
  \hline
\end{tabular}}
\end{table}

\begin{table}[H]
\centering
\caption{\small{Eigenvalues on randomly generated dense matrices for fixed p = 6 and varying n.}}
\label{tab:11}
\resizebox{8.5cm}{!}{\begin{tabular}{|l|c|c|c|c|c|c|}
  \hline
  \textbf{n} & \textbf{50} & \textbf{100} & \textbf{600} & \textbf{1000} & \textbf{2000} & \textbf{3000} \\
  \hline
  \multicolumn{7}{|c|}{ \textbf{Sgmin}  }\\
  \hline
  \textbf{Nitr}	 &23.61       &28.43	 &51.8      &66.26      &86.87      &94.30  \\
  \textbf{Time}	 &0.4731	  &0.8486	 &19.4008	&63.3436	&331.3508   &773.992  \\
  \textbf{Nfe}	 &51.1	      &58.26	 &85.44	    &101.97	    &124.95     &133.37  \\
  \textbf{Fval}	 &9.14e+02	  &2.04e+03	 &1.37e+04	&2.32e+04	&4.70e+04   &7.08e+04  \\
  \textbf{NrmG}	 &6.53e-05	  &1.78e-04	 &1.50e-03	&2.70e-03	&5.60e-03   &8.50e-03  \\
  \textbf{Feasi} &4.14e-16	  &4.27e-16	 &4.74e-16	&5.22e-16	&5.71e-16   &5.92e-16  \\
  \textbf{Error} &2.13e-13	  &3.88e-13	 &2.00e-12	&2.51e-12	&4.50e-12   &5.16e-12 \\
  \hline
  \multicolumn{7}{|c|}{ \textbf{OptStiefel}  }\\
  \hline
  \textbf{Nitr}	 &63.07	      &70.33     &114.05    &135	    &169.01    &188.73  \\
  \textbf{Time}	 &0.0121	  &0.0192	 &0.1829	&0.6169	    &3.2749    &7.3311  \\
  \textbf{Nfe}	 &66.83	      &74.18	 &118.24	&139.91	    &179.01    &200.16  \\
  \textbf{Fval}	 &9.14e+02	  &2.04e+03	 &1.37e+04	&2.32e+04	&4.70e+04  &7.08e+04  \\
  \textbf{NrmG}	 &2.50e-05	  &8.13e-05	 &8.86e-04	&1.40e-03	&3.60e-03  &4.90e-03  \\
  \textbf{Feasi} &6.58e-16	  &8.05e-16	 &1.70e-15	&2.24e-15	&3.29e-15  &3.87e-15  \\
  \textbf{Error} &3.25e-14	  &6.82e-14	 &5.45e-13	&8.40e-13	&2.21e-12  &2.61e-12  \\
  \hline
  \multicolumn{7}{|c|}{ \textbf{Grad-retrac}  }\\
  \hline
  \textbf{Nitr}	 &62.02	      &69.29	 &116.77	&136.82	    &167.55   &187.77   \\
  \textbf{Time}	 &0.0171	  &0.0276	 &0.2073	&0.6494	    &3.2812   &7.3301   \\
  \textbf{Nfe}	 &63.26	      &70.53	 &119.84	&141.45	    &176.35   &197.70   \\
  \textbf{Fval}	 &9.14e+02	  &2.04e+03	 &1.37e+04	&2.32e+04	&4.70e+04 &7.08e+04   \\
  \textbf{NrmG}	 &2.43e-05	  &8.92e-05	 &8.20e-04	&1.70e-03	&3.10e-03 &4.10e-03   \\
  \textbf{Feasi} &1.77e-14	  &1.89e-14	 &1.50e-14	&1.69e-14	&1.81e-14 &2.12e-14   \\
  \textbf{Error} &2.96e-14	  &6.89e-14	 &5.21e-13	&8.06e-13	&1.59e-12 &1.94e-12   \\
  \hline
\end{tabular}}
\end{table}

\section{Conclusion}
In this article we study a feasible approach to deal with orthogonally constrained optimization problems. Our algorithm implements the non-monotone Barzilai-Borwein line search on a mixed gradient based search direction. The feasibility at each iteration is guaranteed by projecting each updated point, which is in the current tangent space, onto the Stiefel manifold, through SVD's decomposition. The mixture is controlled by the coefficients $\alpha$ and $\beta$, associated to $\nabla \mathcal{F}_1(X_k)$ and $\nabla \mathcal{F}_2(X_k)$ respectively. When $\alpha=1$ and $\beta=0$, the obtained direction coincides with Manton's direction, and the difference to these method is just the implementation of BB-line search instead of Armijo's.\\

Our \emph{Grad-retract} algorithm is numerically compared to other state-of-the art algorithms, in a variety of test problems, achieving clear advantages in many cases.

\begin{acknowledgements}
This work was supported in part by CONACYT (Mexico), Grant 258033 and the Brazilian Government, through the  Excellence Fellowship Program  CAPES/IMPA of the second author while visiting the department of Mathematics at UFPR.
\end{acknowledgements}

\end{document}